\documentclass[a4paper,12pt]{article}
\usepackage{latexsym,amsmath,amssymb}
\usepackage{ascmac} 
\usepackage{amscd}
\usepackage{amsfonts}
\usepackage{mathrsfs}
\usepackage{bm} 
\usepackage{enumerate}  
\usepackage{hhline} 
\usepackage{amsthm}
\usepackage{secdot}
\sectiondot{subsection}
\DeclareSymbolFont{lettersA}{U}{txmia}{m}{it}
 \DeclareMathSymbol{\Indi}{\mathord}{lettersA}{'211}

\makeatletter
    
    \@addtoreset{equation}{section}
  \makeatother

 \def\XXint#1#2#3{{\setbox0=\hbox{$#1{#2#3}{\int}$}
 \vcenter{\hbox{$#2#3$}}\kern-.5\wd0}}

\def\Swiech{{\accent"13S}wie{\hbox{\kern -0.21em\lower 0.79ex\hbox{$\textfont1=\scriptfont1\lhook$}}}ch}
\def\SWIECH{{\accent"13S}WIE{\hbox{\kern -0.26em\lower 0.77ex\hbox{$\textfont1=\scriptfont1\lhook$}}}CH}

\def\k{\kappa}
\def\P{{\cal P}}

\def\to{\rightarrow}
\def\R{{\mathbb R}}
\def\N{{\mathbb N}}
\def\l{\lambda}
\def\L{\Lambda}

\def\ti{\times}
\def\le{\left}
\def\ri{\right}
\def\O{\Omega}
\def\o{\omega}

\def\fr{\frac}
\def\p{\partial}
\def\e{\varepsilon}
\def\d{\delta}
\def\s{\sigma}

\def\b{\beta}
\def\a{\alpha}
\def\ol{\overline}

\def\ul{\underline}

\def\ra{\rangle}
\def\la{\langle}

\def\dist{\mathrm{dist}}

\newtheorem{thm}{Theorem}[section]
\newtheorem{dfn}[thm]{Definition}
\newtheorem{lem}[thm]{Lemma}
\newtheorem{prop}[thm]{Proposition}

\newtheorem{rem}[thm]{Remark}

\addcontentsline{toc}{section}{Acknowledgement}

\title{On $L^p$-viscosity solutions of parabolic bilateral
obstacle problems with unbounded ingredients}
\author{
Shota Tateyama\\
Department of Mathematics\\
Waseda University\\
Tokyo 169-8555\\
JAPAN\\
s.tateyama@kurenai.waseda.jp
}
\date{}

\begin{document}
\pagenumbering{roman}
\maketitle
\pagenumbering{arabic}

\begin{abstract}
The global equi-continuity estimate on $L^p$-viscosity solutions of parabolic bilateral obstacle problems with unbounded ingredients is established when obstacles are merely continuous. 
The existence of $L^p$-viscosity solutions is established via an approximation of given data. The local H\"older continuity estimate on the space derivative of $L^p$-viscosity solutions is shown when the obstacles belong to $C^{1, \beta}$, and $p>n+2$. 
\end{abstract}

{\it 
\begin{center}
\begin{tabular}{ll}Keywords:&fully nonlinear parabolic equations, viscosity 
solutions, \\& 
obstacle problems.\vspace{1mm}\\
2010 MSC:&49L25 
  , 35D40 
  ,35K10 
   , 35R35, 
  35B65. 
  \vspace{1mm}\\

\end{tabular}
\end{center}
} 

\section*{Acknowledgements} 
The author would like to thank Professor S. Koike for his helpful comments. 
The author is supported by Grant-in-Aid for Japan Society for Promotion Science Research Fellow 16J02399, in part by Grant-in-Aid for Scientific Research (No. JP16H06339, JP19H05599) of Japan Society for Promotion Science, and by Foundation of Research Fellows, The Mathematical Society of Japan.\section{Introduction}
  \label{sec:intro}
    

  %
  In this paper, we consider the following parabolic bilateral obstacle problem 
\begin{equation}\label{E0} 
\min\le\{\max\le\{u_t+F(x, t, Du, D^2u)-f, u-\psi\ri\}, u-\varphi\ri\}=0\quad\mbox{in }\O_T
\end{equation}
under the Cauchy-Dirichlet condition $u=g$ on $\p_p \O_T$. 
Here, $\O_T:=\O\ti(0, T]$ for a bounded domain $\O\subset\R^n$ and $T>0$, $F$ is at least a measurable function on $\O_T\ti\R^n\ti S^n$, and $f$, $\varphi$, $\psi$ and $g$ are given. We denote $S^n$ by the set of all $n\ti n$ real-valued symmetric matrices with the standard order, and set 
\[
S^n_{\l, \L}:=\le\{ X\in S^n : \l I\leq X \leq \L I\ri\}\quad\mbox{for }0<\l\leq \L.
\]
 Moreover, we denote the parabolic boundary of $\O_T$ by
\[\p_p \O_T:=\O\ti\{0\}\bigcup \p\O\ti[0, T).\]

To begin with, the theory of obstacle problems is motivated by numerous applications, e.g. in stochastic control theory, in economics, in mechanics, in mathematical physics or in mathematical biology. 

An existence theory for parabolic unilateral obstacle problems was first introduced by J.-L. Lions and G. Stampacchia in \cite{LS}. 
 In \cite{Bre}, regularity of solutions of parabolic unilateral obstacle problems was studied by H. Br\'ezis. 
Then, A. Friedman in \cite{F1, F2} considered stochastic games and studied regularity of solutions of bilateral obstacle problems. Afterwards, there appeared numerous researches on parabolic obstacle problems when $F$ are partial differential operators of divergence form. 
 We only refer to \cite{F3, KS, BL, F4, Bar, Rod} and references therein for the existence and regularity of solutions of parabolic obstacle problems and applications. 

In \cite{PS, Sh}, we considered unilateral obstacle problems for fully nonlinear uniformly parabolic operators under appropriate assumptions for applying the regularity theory of viscosity solutions in \cite{W, W2, W3}. It is natural to ask whether the results can be extended to bilateral obstacle problems. We refer to \cite{MR} for an accomplished overview of bilateral obstacle problems.


Although bilateral obstacle problems have been studied since the 1960s, some results on bilateral obstacle problems for non-divergence form operators only have been obtained very recently. In particular, L.F. Duque in \cite{D} showed interior H\"older estimates on viscosity solutions of bilateral obstacle problems for fully nonlinear uniformly parabolic operators with no variable coefficients, no first derivative terms and constant inhomogeneous terms when the obstacles are independent of time and H\"older continuous; 
\[
\begin{cases}
F(x, t, \xi, X)=F(X)\ \mbox{ for } (x, t, \xi, X)\in \O_T\ti\R^n\ti S^n,\\
 f\equiv C, \\
  \varphi(x, t)=\varphi(x),\  \psi(x, t)=\psi(x)\ \mbox{ for } (x, t)\in \O_T,  \\ 
  \varphi, \psi\in C^{\alpha}(\O_T)\ \mbox{ for }\alpha\in(0, 1).
\end{cases}
\]
Under the above hypotheses, in \cite{D}, we obtain the existence of viscosity solutions of $(\ref{E0})$ under the Cauchy-Dirichlet condition, and interior H\"older estimates on the space derivative when the obstacles are in $C^{1, \b}$ for $\b\in(0, 1)$ and separated. 
The corresponding results for elliptic problems are also established in \cite{D}. It was later extended for fully nonlinear uniformly elliptic equations with unbounded coefficients and inhomogeneous terms in \cite{KT}. We will give the definition of $C^{k, \alpha}$ for $k=0$, $1$ and $\alpha\in (0, 1)$ in Section 2.

This paper is the parabolic counterpart of \cite{KT} on fully nonlinear elliptic bilateral obstacle problems. Our aim in this paper is to extend results in \cite{D} when $F$ is a fully nonlinear uniformly parabolic operator. More precisely, under more general hypotheses than those in \cite{D}, we show the equi-continuity of $L^p$-viscosity solutions of (\ref{E0}) in $\ol \O_T$, the existence of $L^p$-viscosity solutions of $(\ref{E0})$, and their local H\"older continuity of space derivatives under additional assumptions. In \cite{D}, it is assumed that the obstacles are separated in order to obtain interior H\"older estimates on the space derivative of viscosity solutions of bilateral obstacle problems. In this paper, we remove this hypothesis (for elliptic case see the Appendix).

Because most of results on the equi-continuity and existence of $L^p$-viscosity solutions of $(\ref{E0})$ follow the same line of arguments as that of its elliptic counterpart used in \cite{KT}, we shall give the outline of proofs. 
As for the local H\"older continuity of space derivatives of $L^p$-viscosity solutions of $(\ref{E0})$, we cannot use our argument used in \cite{KT} because the domain, where the infimum is taken, differs from that of the $L^{\e_0}$ (quasi-) norm in the weak Harnack inequality, which arises in Proposition \ref{paraWeak}. Instead, we use a compactness-based technique developed in \cite{Sh}.

For any $p>0$ and $u:\O_T\to\R$, we denote the quasi-norm: 
$$
\| u\|_{L^p(\O_T)}=\le(\int_0^T\int_\O |u(x, t)|^p \, dx dt\ri)^{\fr1p}.$$
We note that $\|\cdot\|_{L^p(\O_T)}$ satisfies 
\begin{equation}\label{eq:tri}
\| u+v\|_{L^p(\O_T)}\leq C_p\le(\| u\|_{L^p(\O_T)}+\|v\|_{L^p(\O_T)}\ri) \quad \mbox{for some }C_p\geq 1. 
\end{equation}
Notice that we may choose $C_p=1$ when $p\geq 1$.

This paper is organized as follows: In Section 2, we recall the definition of $L^p$-viscosity solutions, basic properties and exhibit main results. Section 3 is devoted to the weak Harnack inequality both in $K\Subset\O_T$ and near $\p_p \O_T$, which yields the global equi-continuity of $L^p$-viscosity solutions. In Section 4, we establish the existence of $L^p$-viscosity solutions of (\ref{E0}) when the obstacles are only continuous under appropriate hypotheses. We obtain H\"older estimates on the space derivative of $L^p$-viscosity solutions in Section 5. 

\section{Preliminaries and main results}
For $(x, t)\in\R^{n+1}$ and $r>0$, we set 
\[B_r:=\{y\in \R^n : |y|<r\},\quad B_r(x):=x+B_r, \]
\[
Q_r:=B_r\ti(-r^2, 0],\quad\mbox{and}\quad Q_r(x, t):=(x, t)+Q_r.  
\]
For any measurable set $A\subset \R^{n+1}$, we denote by $|A|$ the $(n+1)$-dimensional Lebesgue measure of $A$. The parabolic distance is defined by 
\[
d((x, t), (y, s)):=\sqrt{|x-y|^2+|t-s|}.
\]
For $U$, $V\subset\R^{n+1}$, we define the distance between $U$ and $V$ by
\[
\dist(U, V):=\inf\le\{ d((x, t), (y, s)) : (x, t)\in U, (y, s)\in V\ri\}. 
\]
In what follows, $K\Subset \O_T$ means that $K\subset \O_T$ is a compact set satisfying $\dist(K, \p_p \O_T)>0$. 

We denote by $C^{2, 1}(\O_T)$ the space of functions $u\in C(\O_T)$ such that $u_t$, $\fr{\p u}{\p x_k}$, $\fr{\p^2u}{\p x_\ell\p x_k}\in C(\O_T)$ for $1\leq k, \ell\leq n$. For $1\leq p\leq \infty$, 
we denote by $W^{2, 1}_{p}(\O_T)$ the space of functions $u\in L^p(\O_T)$ such that $u_t$, $\fr{\p u}{\p x_k}$, $\fr{\p^2u}{\p x_\ell\p x_k}\in L^p(\O_T)$ for $1\leq k, \ell\leq n$.

By following notations from \cite{CKS}, for any $\tilde \O_T\subset\R^{n+1}$ such that $\O_T\subset\tilde\O_T\subset\ol \O_T$, and $\alpha\in(0, 1)$, the spaces $C^{0,\alpha}(\tilde \O_T)$ and $C^{1, \alpha}(\tilde\O_T)$ denote the set of all functions $u$ defined in $\tilde\O_T$ satisfying 
\[\|u\|_{C^{0,\alpha}(\tilde\O_T)}:=\|u\|_{L^\infty(\tilde\O_T)}+\sup_{\substack{(x, t), (y, s)\in \tilde\O_T\\ (x, t)\neq (y, s)}}\fr{|u(x, t)-u(y, s)|}{d((x, t), (y, s))^\alpha}<\infty
\]
and 
\begin{align*}
\|u\|_{C^{1, \alpha}(\tilde\O_T)}:=&\|u\|_{L^\infty(\tilde\O_T)}+\|Du\|_{L^\infty(\tilde\O_T)}\\
&+\sup_{\substack{(x, t), (y, s)\in \tilde\O_T\\ (x, t)\neq (y, s)}}\fr{|u(y, s)-u(x, t)-\langle Du(x, t), y-x\rangle|}{d((x, t), (y, s))^{1+\alpha}}<\infty,
\end{align*}
respectively. 
In what follows, we simply write $C^\alpha(\tilde\O_T)$ for $C^{0,\alpha}(\tilde\O_T)$.

We recall the definition of $L^p$-viscosity solutions of general parabolic partial differential equations (PDE for short) from \cite{CKS}: 
\begin{equation}\label{Gpara}
u_t+G(x, t, u, Du, D^2u)=0 \quad\mbox{in }\O_T,
\end{equation}
where $G:\O_T\ti\R\ti \R^n\ti S^n\to\R$ is a measurable function.    
\begin{dfn}
We say that a function $u\in C(\O_T)$ is an $L^p$-viscosity subsolution (resp., supersolution) of $(\ref{Gpara})$ when $u$ satisfies for any $\eta\in W^{2, 1}_{p}(\O_T)$, 
\[
\lim_{r\to 0} ess \inf_{Q_r(x_0, t_0)} \{\eta_t(x, t)+G(x, t, u(x, t), D\eta(x, t), D^2\eta(x, t))\}\leq 0
\]
\[
\le(resp.,\ \lim_{r\to 0} ess \sup_{Q_r(x_0, t_0)} \{\eta_t(x, t)+G(x, t, u(x, t), D\eta(x, t), D^2\eta(x, t))\}\geq 0\ri)
\]
provided that $u-\eta$ attains its local maximum (resp., minimum) at $(x_0, t_0)\in \O_T$. We say that $u$ is an $L^p$-viscosity solution of $(\ref{Gpara})$ when $u$ is an $L^p$-viscosity subsolution and an $L^p$-viscosity supersolution of $(\ref{Gpara})$. 
\end{dfn}

\begin{rem}
We will call $C$-viscosity subsolutions (resp., supersolutions, solutions) if we replace $W^{2,1}_{p}(\O_T)$ by $C^{2, 1}(\O_T)$ in the above when given $G$ is continuous. 
We refer to \cite{CIL} for the theory of $C$-viscosity solutions. 
\end{rem}

In order to present main results, we shall prepare some notations and hypotheses. 
Throughout this paper, under the hypothesis 
\begin{equation}\label{paraApq}
p_1<p\leq q, \quad q>n+2,
\end{equation}
where $p_1=p_1(n, \fr\L\l)\in [\fr{n+2}{2}, n+1)$ is the constant in \cite{Es}, we suppose 
\begin{equation}\label{paraAf}
f\in L^p(\O_T).
\end{equation}
The structure condition on $F$ is that there exist constants $0<\l\leq\L$ and 
\begin{equation}\label{paraAmu}
\mu\in L^q(\O_T) 
\end{equation}
such that 
\begin{equation}\label{paraAF}
\begin{cases}
&F(x, t, 0, O)=0 \quad\mbox{and}\\
&\P^-_{\l, \L}(X-Y)-\mu(x, t)|\xi-\zeta|\leq F(x, t, \xi, X)-F(x, t, \zeta, Y)\\
&\leq \P^+_{\l, \L}(X-Y)+\mu(x, t)|\xi-\zeta|
\end{cases}
\end{equation}
for $(x, t)\in \O_T$, $\xi$, $\zeta\in\R^n$, $X$, $Y\in S^n$, where $\P^\pm_{\l. \L}:S^n\to\R$ are defined by 
\[
\P^+_{\l, \L}(X):=\max\{-\mathrm{Tr}(AX) : A\in S^n_{\l, \L}\},\quad\mbox{and}\quad
\P^-_{\l, \L}(X):=-\P^+_{\l, \L}(-X)
\]
for $X\in S^n$. Because we fix $0<\l\leq \L$ in this paper, we shall write $\P^\pm:=\P^\pm_{\l, \L}$ for simplicity. 
 We note that (\ref{paraAF}) implies 
$\mu\geq0$ in $\O_T.$

For obstacles $\varphi$ and $\psi$, as compatibility conditions, we suppose that 
\begin{equation}\label{paraAob} 
  \varphi\leq \psi\quad\mbox{in }\O_T.
\end{equation}

\subsection{Basic properties}

We first give a direct consequence from the definition, which is a modification of Proposition 2.3 in \cite{KT}.  
 Hereafter, for $\a$, $\b\in\R$, we use 
 \[\a \vee \b:=\max\{\a, \b\},\, \a \wedge \b:=\min\{\a, \b\},\, \a^+:=\a \vee 0 \,\mbox{ and }\, \a^-:=(-\a)\vee 0.\]  
 
 \begin{prop}\label{prop:def}
Assume $(\ref{paraApq})$, $(\ref{paraAf})$, $(\ref{paraAmu})$, $(\ref{paraAF})$ and $(\ref{paraAob})$. Let $u\in C(\O_T)$ be an $L^p$-viscosity subsolution (resp., supersolution) of $(\ref{E0})$. 
Assume that an affine function $\ell(x):=a+\langle b, x\rangle$, 
where $(a, b)\in\R\ti\R^n$, 
satisfies that $\ell\geq\varphi$ $\le(\mbox{resp., } \ell\leq \psi\ri)$ in a relatively open subset $Q\subset \O_T.$
Then, 
$u\vee \ell$ $(\mbox{resp., } u\wedge \ell)$ is an $L^p$-viscosity subsolution (resp., supersolution) of 
\[
u_t+\P^-(D^2u)-\mu|Du|-f^+=0 \quad \mbox{in }Q
\]
\[\le(\mbox{resp., } u_t+\P^+(D^2u)+\mu|Du|+f^-=0\quad \mbox{in }Q\ri).
\]
\end{prop}

\begin{proof}
We only give a proof for supersolutions. 

For $\eta\in W^{2,1}_{p}(Q)$, we suppose that $(u\wedge \ell)-\eta$ attains its local minimum at $(x_0, t_0)\in Q$. 

If $u(x_0, t_0)<\ell(x_0)$ holds, then $u-\eta$ attains its local minimum at $(x_0, t_0)\in Q$, and $u<\psi$ near $(x_0, t_0)$. 
Thus, by the definition, we have 
\[
\lim_{r\to 0} ess \sup_{Q_r(x_0, t_0)} \min\le\{\eta_t+F(x, t, D\eta, D^2\eta)-f, u-\varphi\ri\}\geq 0, 
\]
which gives the assertion by $(\ref{paraAF})$. 

When $u(x_0, t_0)\geq \ell(x_0)$, we only note that $\ell$ is an $L^p$-viscosity supersolution of $u_t+\P^+(D^2u)=0$ in $\O_T$.
\end{proof}

We shall introduce a scaled version of the weak Harnack inequality and a H\"older continuity in \cite{KST}. Modifying the result in \cite{KST} by an argument of the compactness, we state the next proposition as simple as possible for later use. See \cite{KST} for the original version. Hereafter, we use 
the notation 
\[ \alpha_0:=2-\fr{n+2}{p\wedge (n+2)}\in(0, 1].\]
\begin{prop} (cf. Theorem 3.4 in \cite{KST})\label{paraWeak}
Let $r>0$. Under the hypothesis $(\ref{paraApq})$, we assume $\mu\in L^q(Q_{2r})$. Then, there exist constants $\e_0>0$ and $C_0>0$ such that for any $f\in L^p(Q_{2r})$ and any nonnegative $L^p$-viscosity supersolution $u\in C(Q_{2r})$ of 
\begin{equation}\label{eq:P+}
u_t+\P^+(D^2u)+\mu|Du|-f=0\quad\mbox{in }Q_{2r}, 
\end{equation}
we have
\begin{equation}\label{pppWHI}
\|u\|_{L^{\e_0}(Q_r(0, -3r^2))}\leq C_0r^{\fr{n+2}{\e_0}}\le(\inf_{Q_r} u+r^{\alpha_0}\|f\|_{L^{p\wedge(n+2)}(Q_{2r})}\ri).
\end{equation}
Here, $\e_0$ and $C_0$ depend only on $n$, $\L$, $\l$, $p$, $q$ and $r^{1-\fr{n+2}{q}}\|\mu\|_{L^q(Q_{2r})}$. 
\end{prop}

We next recall how to derive a local H\"older estimate on $L^p$-viscosity solutions of parabolic extremal equations in order to show a key idea of this paper. The following proposition is a scaled version of Theorem 4.2 in \cite{KST}.  
\begin{prop}[cf. Theorem 4.2 in \cite{KST}]\label{prop:Holder}
Let $R>0$. Under $(\ref{paraApq})$, we assume $\mu\in L^q(Q_{2R})$. 
Then, there exist constants $K_1>0$ and $\hat{\alpha}\in(0, \alpha_0]$ such that for any $f\in L^p(Q_{2R})$, if $u\in C(Q_{2R})$ is an $L^p$-viscosity subsolution and an $L^p$-viscosity supersolution, respectively, of   
\[
u_t+\P^-(D^2u)-\mu|Du|-f=0 \quad\mbox{and}\quad u_t+\P^+(D^2u)+\mu|Du|-f=0
\]
in $Q_{2R}$, then it follows that 
\[
|u(x, t)-u(y, s)|\leq K_1\biggl(\fr{d((x, t), (y, s))}{R}\biggr)^{\hat\alpha}\le(\|u\|_{L^\infty(Q_{2R})}+R^{\alpha_0}\|f\|_{L^{p\wedge(n+2)}(Q_{2R})}\ri)
\]
for $(x, t), (y, s)\in Q_{R}$. 
Here, $K_1$ and $\hat\alpha$ depend only on $n, \L, \l, p, q$ and $R^{1-\fr{n+2}{q}}\|\mu\|_{L^q(Q_{2R})}$. 
\end{prop}

\begin{proof}
Fix $(x, t)\in Q_{R}$. 
For $0<s\leq R$, we set 
\[
M_{s}:=\sup_{Q_{s}(x, t)}u, \quad\mbox{and} \quad m_s:=\inf_{Q_{s}(x, t)} u.
\]
Now, for $0<r\leq \fr R2$, setting 
\[
U:=u-m_{2r}\geq 0, \quad\mbox{and}\quad V:=M_{2r}-u\geq 0\quad\mbox{in }Q_{2r}(x, t),
\]
we immediately see that $U$ and $V$ are $L^p$-viscosity supersolutions of (\ref{eq:P+}) in $Q_{2r}(x, t)$ with $f$ replaced by $-f^-$ and $-f^+$, respectively. Hence, in view of Proposition \ref{paraWeak}, we have 
\[
\|U\|_{L^{\e_0}(Q_r(x, t-3r^2))}\leq C_0 r^{\fr{n+2}{\e_0}}\le(\inf_{Q_r(x, t)} U+r^{\alpha_0}\|f\|_{L^{p\wedge(n+2)}(Q_{2r}(x, t))}\ri),
\]
\[
\|V\|_{L^{\e_0}(Q_r(x, t-3r^2))}\leq C_0 r^{\fr{n+2}{\e_0}}\le(\inf_{Q_r(x, t)} V+r^{\alpha_0}\|f\|_{L^{p\wedge(n+2)}(Q_{2r}(x, t))}\ri). 
\]
Therefore, by (\ref{eq:tri}), these inequalities imply 
\begin{align*}
M_{2r}-m_{2r}&=|Q_r|^{-\fr{1}{\e_0}}\|M_{2r}-m_{2r}\|_{L^{\e_0}(Q_r(x, t-3r^2))}\\
&\leq |Q_r|^{-\fr{1}{\e_0}}C_{\e_0}\le(\|V\|_{L^{\e_0}(Q_r(x, t-3r^2))}+\|U\|_{L^{\e_0}(Q_r(x, t-3r^2))}\ri)\\
&\leq C'_0\le(M_{2r}-M_r+m_r-m_{2r}+2r^{\alpha_0}\|f\|_{L^{p\wedge(n+2)}(Q_{2R})}\ri),
\end{align*}
where $C'_0:=|Q_1|^{-\fr{1}{\e_0}}C_{\e_0}C_0$. 
Thus, there exists $\theta_0\in(0, 1)$ such that 
\[
\o(r)\leq \theta_0\o(2r)+2r^{\alpha_0}\|f\|_{L^{p\wedge(n+2)}(Q_{2R})}, 
\]
where $\o(r):=M_r-m_r$. Hence, the standard argument (e.g. Lemma 8.23 in \cite{GilTru83}) implies that 
\[
|u(x, t)-u(y, s)|\leq K_1\biggl(\fr{d((x, t), (y, s))}{R}\biggr)^{\hat{\alpha}}\le(\|u\|_{L^\infty(Q_{2R})}+R^{\alpha_0}\|f\|_{L^{p\wedge(n+2)}(Q_{2R})}\ri)
\]
for some $K_1>0$ and $\hat{\alpha}\in(0, \alpha_0]$. 
\end{proof}

\begin{rem}
One of key ideas of this paper is a different choice of $M_s$ and $m_s$ in the above for the proof of Lemma 
	\ref{paralemlocal}.
	\end{rem}

We introduce the H\"older continuity of the space derivative for $C$-viscosity solutions of fully nonlinear uniformly parabolic PDE. 

\begin{prop}\label{prop:C1}(cf. Theorem 4.8 in \cite{W2}, Proposition 5.4 in \cite{CKS})
Assume that $F: S^n\to\R$ satisfies
\[\P^-(X-Y)\leq F(X)-F(Y)\leq\P^+(X-Y) \]
for $X, Y\in S^n$. Then, there exist constants $K_2>0$ and $\hat{\beta}\in(0, 1)$, depending only on $n$, $\L$ and $\l$, such that if $u\in C(Q_1)$ is a $C$-viscosity solution of $u_t+F(D^2u)=0$ in $Q_1$, then it follows that  
\[
\|u\|_{C^{1, \hat{\beta}}(\ol{Q}_{\fr12})}\leq K_2\|u\|_{L^\infty(Q_1)}. 
\]
\end{prop}

We finally give a reasonable property of $L^p$-viscosity solutions of (\ref{E0}), which will be often used without mentioning it. Because the proof follows by the arguments of the proof of Proposition 2.9 in \cite{KT}, we omit it. 

\begin{prop}
\label{prop:bounded} (cf. Proposition 2.9 in \cite{KT})
Under $(\ref{paraApq})$, $(\ref{paraAf})$, $(\ref{paraAmu})$, $(\ref{paraAF})$ and $(\ref{paraAob})$, we assume $\varphi$, $\psi\in C(\O_T)$. Then, for any $L^p$-viscosity subsolution (resp., supersolution) $u\in C(\O_T)$ of $(\ref{E0})$, we have 
\[
u\leq\psi \quad(\mbox{resp., } u\geq \varphi) \quad \mbox{in }\O_T. 
\]
\end{prop}

\subsection{Main results}

For obstacles, under $(\ref{paraAob})$, we at least assume 
\begin{equation}\label{paraAconti}
\varphi, \psi\in C(\ol \O_T). 
\end{equation} 

In order to obtain the estimate near $\p_p\O_T$, we suppose the following condition on the shape of $\O$, which was introduced in \cite{BNV}: 
\begin{equation}\label{paraAboundary}
\le\{\begin{array}{l}
\mbox{There exist } R_0>0\mbox{ and } \Theta_0>0 \mbox{ such that }\\
|B_r(x)\setminus\O|\geq \Theta_0 r^n \mbox{ for } (x, r)\in \p\O\ti(0, R_0).
\end{array}\ri.
\end{equation}
For the Cauchy-Dirichlet datum, we suppose that 
\begin{equation}\label{paraCDdata}
  g\in C(\p_p \O_T),\quad\mbox{and}\quad\varphi\leq g\leq \psi\quad\mbox{on }\p_p \O_T. 
\end{equation}

We call a function $\o:[0, \infty)\to[0, \infty)$ a modulus of continuity if $\o$ is nondecreasing and continuous in $[0, \infty)$ such that $\o(0)=0$. 

Our first result is the global equi-continuity estimate on $L^p$-viscosity solutions of (\ref{E0}). We present a proof of the following theorem in Section 3. 

\begin{thm}\label{cor:Global}
Assume $(\ref{paraApq})$, $(\ref{paraAf})$, $(\ref{paraAmu})$, $(\ref{paraAF})$, $(\ref{paraAob})$, $(\ref{paraAconti})$, $(\ref{paraAboundary})$ and $(\ref{paraCDdata})$. Then, there exists a modulus of continuity $\omega_0$ such that if $u\in C(\ol \O_T)$ is an $L^p$-viscosity solution of $(\ref{E0})$ satisfying \begin{equation}\label{paraBC}
u=g \quad\mbox{on } \p_p \O_T,
\end{equation} then it follows that 
\[
|u(x, t)-u(y, s)|\leq \omega_0(d((x, t), (y, s)))\quad\mbox{for } (x, t), (y, s)\in \ol \O_T. 
\]
Moreover, if we assume that
\[
\varphi, \psi\in C^{\alpha_1}(\ol \O_T), \quad\mbox{and}\quad g\in C^{\alpha_1}(\p_p \O_T)\quad\mbox{for } \alpha_1\in(0, 1),
\]
then there exist $\alpha_2\in (0, \alpha_0\wedge\alpha_1]$ and $C>0$, independent of $u$, such that 
\[
|u(x, t)-u(y, s)|\leq Cd((x, t), (y, s))^{\alpha_2}\quad\mbox{for } (x, t), (y, s)\in \ol \O_T.
\]
\end{thm}

 Thanks to Theorem \ref{cor:Global}, we establish the following existence result whose proof is presented in Section 4.  

\begin{thm}\label{thm:paraEx}
Under $(\ref{paraApq})$, $(\ref{paraAf})$, $(\ref{paraAmu})$, $(\ref{paraAF})$, $(\ref{paraAob})$, $(\ref{paraAconti})$ and $(\ref{paraCDdata})$, we assume the uniform exterior cone condition on $\O$. Then, there exists an $L^p$-viscosity solution $u\in C(\ol \O_T)$ of $(\ref{E0})$ satisfying $(\ref{paraBC})$.
\end{thm}

In Section 5, assuming 
\begin{equation}\label{paraApq2}
q\geq p>n+2,
\end{equation}
we define 
\[
\beta_0:=1-\fr{n+2}{p}\in(0, 1). 
\]
To obtain $C^{1, \beta}$ estimates on $L^p$-viscosity solutions of (\ref{E0}), we suppose that 
\begin{equation}\label{paraAregob}
\varphi, \psi\in C^{1, \beta_1}(\O_T)\quad\mbox{for }\beta_1\in (0, 1).
\end{equation}
In Section 5.2, we will use the constant $\beta_2$ defined by 
\[
\beta_2:=\beta_0\wedge\beta_1\in(0, 1).
\]

In order to state the next theorem, we prepare some notations. 
For small $r>0$, we introduce subdomains of $\O$ and $\O_T$, respectively, 
\begin{equation*}
\O^r:=\{x\in \O : \dist(x, \p\O)>r\},\quad\mbox{and}\quad
\O_T^r:=\O^r\ti(r^2, T].
\end{equation*}

For $u\in C(\O_T)$ such that $\varphi\leq u \leq \psi$ in $\O_T$, we set coincidence sets
\begin{align*}
C^-[u]&:=\{(x, t)\in \O_T : u(x, t)=\varphi(x, t)\},\\
 C^+[u]&:=\{(x, t)\in \O_T : u(x, t)=\psi(x, t)\},\\
C^\pm[u]&:=C^-[u]\cup C^+[u]\subset \O_T, 
 \end{align*}
and the non-coincidence set
\[
N[u]:=\O_T\setminus C^\pm[u]=\{(x, t)\in \O_T : \varphi(x, t)<u(x, t)<\psi(x, t)\}.
\]
For small $r>0$, we define subdomains of $N[u]$ 
\[ N_r[u]:=\{(x, t)\in \O^r_T : \dist((x, t), C^\pm[u])>r\}.\] 
 
For $F$ in (\ref{E0}), we use the following notation: 
\[
\theta((x, t), (y, s)):=\sup_{X\in S^n}\fr{|F(x, t, 0, X)-F(y, s, 0, X)|}{1+\|X\|}\quad\mbox{for }(x, t), (y, s)\in \O_T.
\]

\begin{thm}\label{Nearcoin}
Assume $(\ref{paraApq2})$, $(\ref{paraAf})$, $(\ref{paraAmu})$, $(\ref{paraAF})$, $(\ref{paraAob})$, and $(\ref{paraAregob})$. 
Let $\beta_3\in (0, \hat{\beta}\wedge\beta_0)\cap(0, \beta_1]$,
where $\hat\beta\in(0, 1)$ is from Proposition $\ref{prop:C1}$.
For each small $\e>0$, there exist $C>0$ and $\delta_0>0$ such that if $u\in C(\O_T)$ is an $L^p$-viscosity solution of $(\ref{E0})$, and if 
\begin{equation}\label{paraAF2}
\fr{1}{r}\|\theta((y, s), \cdot)\|_{L^{n+2}(Q_r(y, s))}\leq \delta_0 \quad\mbox{for } r\in (0, \e]\mbox{ and }(y, s)\in N_{\e}[u],
\end{equation} 
then it follows that 
\[
|Du(x, t)-Du(y, s)|\leq Cd((x, t), (y, s))^{\beta_3}\quad\mbox{for }(x, t), (y, s)\in \O_T^\e. 
\]
\end{thm}

\section{Global equi-continuity estimates}

In what follows, under (\ref{paraAconti}), we denote by $\sigma_0$ the modulus of continuity of $\varphi$ and $\psi$ in $\ol \O_T$:
\[
\sigma_0(r):=\sup\biggl\{|\varphi(x, t)-\varphi(y, s)|\vee|\psi(x, t)-\psi(y, s)|\le.:\begin{array}{l} (x, t), (y, s)\in \O_T\\
d((x, t), (y, s))<r	
\end{array}\ri.
\biggr\}. 
\]

\subsection{Local estimates}
We first show the local equi-continuity estimate on $L^p$-viscosity solutions of (\ref{E0}). 

\begin{lem}\label{paralemlocal}
Assume $(\ref{paraApq})$, $(\ref{paraAf})$, $(\ref{paraAmu})$, $(\ref{paraAF})$, $(\ref{paraAob})$ and $(\ref{paraAconti})$. For any compact set $K\Subset \O_T$, there exists a modulus of continuity $\omega_0$ such that if $u\in C(\O_T)$ is an $L^p$-viscosity solution of $(\ref{E0})$, then it follows that 
\[
|u(x, t)-u(y, s)|\leq \omega_0(d((x, t), (y, s)))\quad\mbox{for }(x, t), (y, s)\in K. 
\] 
\end{lem}

\begin{proof}
Let $r\in(0, \fr{d_0}{2}]$, where $d_0:=\dist(K, \p_p \O_T)$, and $(x_0, t_0)\in K$. Considering $u(x+x_0, t+t_0)$, we may suppose that $(x_0, t_0)=(0, 0)$. 
Setting $\sigma_0:=\sigma_0(2\sqrt{2}r)$, we define 
\[
\ol{u}:=u\vee(\varphi(0, 0)+\sigma_0)\quad\mbox{and}\quad \ul{u}:=u\wedge(\psi(0, 0)-\sigma_0). 
\]
By noting that $\varphi(0, 0)+\sigma_0\geq\varphi$ and $\psi\geq \psi(0, 0)-\sigma_0$ in $Q_{2r}$, it follows from Proposition \ref{prop:def} that $\ol{u}$ and $\ul{u}$ are, respectively, an $L^p$-viscosity subsolution and an $L^p$-viscosity supersolution of 
\[
u_t+\P^-(D^2u)-\mu|Du|-f^+=0\quad\mbox{and}\quad u_t+\P^+(D^2u)+\mu|Du|+f^-=0 
\]
in $Q_{2r}$. 

Now, for $s\in(0, d_0]$, setting 
\[
M_s:=\sup_{Q_s} \,\ol{u},\quad\mbox{and}\quad m_s:=\inf_{Q_s} \,\ul{u},
\]
we define
\[
U:=\ul{u}-m_{2r}, \quad\mbox{and}\quad V:=M_{2r}-\ol{u}
\] for $r\in(0, \fr{d_0}{2}]$. 

It is easy to see that $U$ and $V$ are, respectively, nonnegative $L^p$-viscosity supersolutions of 
\[
u_t+\P^+(D^2u)+\mu|Du|+f^\mp=0\quad\mbox{in }Q_{2r}. 
\]
Hence, by Proposition \ref{paraWeak}, we have
\begin{equation}\label{parawH1}
\|U\|_{L^{\e_0}(Q_r(0, -3r^2))}\leq Cr^{\fr{n+2}{\e_0}}\le(\inf_{Q_r}\, U+r^{\alpha_0}\|f\|_{L^{p\wedge(n+2)}(Q_{2r})}\ri)
\end{equation}
and 
\begin{equation}\label{parawH2}
\|V\|_{L^{\e_0}(Q_r(0, -3r^2))}\leq Cr^{\fr{n+2}{\e_0}}\le(\inf_{Q_r}\, V+r^{\alpha_0}\|f\|_{L^{p\wedge(n+2)}(Q_{2r})}\ri).
\end{equation}
Hereafter, $C>0$ denotes the various constant depending only on known quantities. 
Since $M_{2r}-m_{2r}=V+(\ol{u}-u)+(u-\ul{u})+U\leq V+4\sigma_0+U$ in $Q_{2r}$ by Proposition \ref{prop:bounded}, we have 
\[
M_{2r}-m_{2r}\leq Cr^{-\fr{n+2}{\e_0}}\le(\|V\|_{L^{\e_0}(Q_r(0, -3r^2))}+\sigma_0 r^{\fr{n+2}{\e_0}}+\|U\|_{L^{\e_0}(Q_r(0, -3r^2))}\ri). 
\]
Combining this with (\ref{parawH1}) and (\ref{parawH2}), we can find $\theta_0\in(0, 1)$ such that 
\[
M_r-m_r\leq \theta_0(M_{2r}-m_{2r})+r^{\alpha_0}\|f\|_{L^{p\wedge(n+2)}(\O_T)}+\sigma_0(2\sqrt{2}r). 
\]
We note here that 
\[u(x, t)-u(y, s)\leq \ol{u}(x, t)-\ul{u}(y, s)\quad \mbox{for } (x, t), (y, s)\in Q_{2r}.
\]
Thus, as in the proof of Proposition \ref{prop:Holder}, in view of Lemma 8.23 in \cite{GilTru83}, it is standard to find a modulus of continuity $\o_0$ in the assertion. 
\end{proof}

\begin{rem}\label{rem:holder}
As noted in Section 2.2, if we suppose $\varphi, \psi\in C^{\alpha_1}(\O_T)$ for $\alpha_1\in(0, 1)$, then we can show $u\in C^{\alpha_2}(\O_T)$ for some $\alpha_2\in\le(0, \alpha_0\wedge\alpha_1\ri]$ because we can choose $\sigma_0(r)=C r^{\alpha_1}$ for some $C>0$ in the above. 
\end{rem}

\subsection{Equi-continuity near $\p_p \O_T$}
We next prove that $u$ is equi-continuous near $\p_p \O_T$.

\begin{lem}\label{paralemboundary}
Assume $(\ref{paraApq})$, $(\ref{paraAf})$, $(\ref{paraAmu})$, $(\ref{paraAF})$, $(\ref{paraAob})$, $(\ref{paraAconti})$, $(\ref{paraAboundary})$ and $(\ref{paraCDdata})$. For small $\e>0$, there exists a modulus of continuity $\omega_0$ such that if $u\in C(\ol \O_T)$ is an $L^p$-viscosity solution of $(\ref{E0})$ satisfying $(\ref{paraBC})$, then it follows that 
\[
|u(x, t)-u(y, s)|\leq \omega_0(d((x, t), (y, s)))\quad\mbox{for }(x, t), (y, s)\in \ol \O_T\setminus\O_T^\e. 
\] 
\end{lem}

\begin{proof}
Let $(x_0, t_0)\in \p_p \O_T$. For simplicity, we may suppose $x_0=0\in\ol\O$ by translation. 

$\ul{\mbox{Case I}: 0\in \p\O \mbox{ and } 0<t_0 \leq T.}$ Let $0<r\leq \fr12$. As in the proof of Lemma \ref{paralemlocal}, we set 
\[
\ol{u}:=u\vee(\varphi(0, t_0)+\sigma_0)\quad\mbox{and}\quad \ul{u}:=u\wedge(\psi(0, t_0)-\sigma_0),
\]
where $\sigma_0:=\sigma\le(2\sqrt{2}r\ri)$. 
In view of Proposition \ref{prop:def} again, we see that $\ol{u}$ and $\ul{u}$ are, respectively, an $L^p$-viscosity subsolution and an $L^p$-viscosity supersolution of 
\begin{align*}
u_t+\P^-(D^2u)-\mu|Du|-f^+=0,\quad\mbox{and}\quad u_t+\P^+(D^2u)+\mu|Du|+f^-=0
\end{align*}
in $Q_{2r}(0, t_0)\cap \O_T$. 
Now, as in \cite{GilTru83, KS3} for instance, for $s\in(0, 1]$, setting 
\[
M_s:=\sup_{Q_{s}(0, t_0)\cap \O_T} \ol{u},\quad\mbox{and}\quad m_s:=\inf_{Q_{s}(0, t_0)\cap \O_T} \ul{u},
\]
we define 
\[
U:=
\le\{\begin{array}{ll}
(\ul{u}-m_{2r})\wedge \ul{c}&\mbox{in } Q_{2r}(0, t_0)\cap \O_T, \\
\ul{c}&\mbox{in } Q_{2r}(0, t_0)\setminus \O_T, 
\end{array}\ri.
\]
and
\[
V:=
\le\{\begin{array}{ll}
(M_{2r}-\ol{u})\wedge \ol{c}&\mbox{in } Q_{2r}(0, t_0)\cap \O_T, \\
\ol{c}&\mbox{in } Q_{2r}(0, t_0)\setminus \O_T, 
\end{array}\ri.
\]
where nonnegative constants $\ul{c}$ and $\ol{c}$ are given by 
\[
\ul{c}:=\inf_{Q_{2r}(0, t_0)\cap \p_p \O_T} \ul{u}-m_{2r},\quad\mbox{and}\quad \ol{c}:=M_{2r}-\sup_{Q_{2r}(0, t_0)\cap \p_p \O_T} \ol{u}. 
\]
Hence, it is easy to verify that $U$ and $V$ are nonnegative $L^p$-viscosity supersolutions of 
\[
u_t+\P^+(D^2u)+\hat{\mu}|Du|+|\hat{f}|=0\quad\mbox{in }Q_{2r}(0, t_0),
\]
where $\hat{f}$ and $\hat{\mu}$ are zero extensions of $f$ and $\mu$ outside of $\O_T$, respectively. In view of Proposition \ref{paraWeak},
 we have 
 \[
\Theta_0^{\fr{1}{\e_0}}\le(\inf_{Q_{2r}(0, t_0)\cap \p_p \O_T} \ul{u}-m_{2r}\ri)\leq C\le(\inf_{Q_{r}(0, t_0)\cap \O_T} \ul{u}-m_{2r} +r^{\alpha_0}\|f\|_{L^{p\wedge(n+2)}(\O_T)}\ri), 
\] 
 and
\[
\Theta_0^{\fr{1}{\e_0}}\le(M_{2r}-\sup_{Q_{2r}(0, t_0)\cap \p_p \O_T} \ol{u}\ri)\leq C\le(M_{2r}-\sup_{Q_{r}(0, t_0)\cap \O_T} \ol{u} +r^{\alpha_0}\|f\|_{L^{p\wedge(n+2)}(\O_T)}\ri).
\]
These inequalities imply that there exists $\theta_0\in(0, 1)$ such that 
\[
M_r-m_r\leq \theta_0(M_{2r}-m_{2r})+2r^{\alpha_0}\|f\|_{L^p(\O_T)}+\sigma_0(2\sqrt{2}r)+\omega_g(2\sqrt{2}r),
\] 
where 
\[\omega_g(r):=\sup\biggl\{|g(x, t)-g(y, s)|\le.:\begin{array}{l} (x, t), (y, s)\in \p_p \O_T\\
d((x, t), (y, s))<r	
\end{array}\ri.
\biggr\}. 
\] 
Therefore, we can find a modulus of continuity $\omega_0$ such that 
\[
|u(x, t)-u(y, s)|\leq \omega_0(d((x, t), (y, s)))\quad\mbox{for } (x, t), (y, s)\in (\ol{\O}\setminus \O^\e)\ti[0, T].
\]

$\ul{\mbox{Case II}: t_0=0.}$ Let $r\in (0, \e]$. As in Case I, setting 
\[
\ol{u}:=u\vee(\varphi(0, 0)+\sigma_0),\quad\mbox{and}\quad 
\ul{u}:=u\wedge(\psi(0, 0)-\sigma_0) \quad\mbox{in }Q_{2r}(0, r^2),
\]
where $\sigma_0:=\sigma_0\le(2\sqrt{2}r\ri)$, we have
\[
\inf_{Q_{2r}(0, r^2)\cap \p_p \O_T} \ul{u}-m_{2r}\leq C\le(\inf_{Q_r(0, r^2)\cap \O_T} \ul{u}-m_{2r} +r^{\alpha_0}\|f\|_{L^{p\wedge(n+2)}(\O_T)}\ri)
\]
and
\[
M_{2r}-\sup_{Q_{2r}(0, r^2)\cap \p_p \O_T} \ol{u}\leq C\le(M_{2r}-\sup_{Q_r(0, r^2)\cap \O_T} \ol{u} +r^{\alpha_0}\|f\|_{L^{p\wedge(n+2)}(\O_T)}\ri),
\]
where
\[
M_s:=\sup_{Q_{s}(0, \fr{s^2}{4})\cap \O_T} \ol{u},\quad\mbox{and}\quad m_s:=\inf_{Q_{s}(0, \fr{s^2}{4})\cap \O_T} \ul{u}
\]
for $s\in(0, 2\e]$. Therefore, as in the Case I, we can find a modulus of continuity $\omega_0$ such that 
\[
|u(x, t)-u(y, s)|\leq \omega_0(d((x, t), (y, s)))\quad\mbox{for } (x, t), (y, s)\in \ol\O\ti[0, \e^2].
\]\end{proof}

\begin{rem}
As in Remark \ref{rem:holder}, if we suppose that $\varphi$, $\psi\in C^{\alpha_1}(\ol \O_T \setminus \O_T^{2\e})$, and $g\in C^{\alpha_1}(\p_p \O_T)$ for $\alpha_1\in(0, 1)$, then $u\in C^{\alpha_2}(\ol \O_T \setminus \O_T^\e)$ holds for some $\alpha_2\in(0, \alpha_0\wedge\alpha_1]$
\end{rem}

\begin{proof}[Proof of Theorem \ref{cor:Global}]
In view of Lemma \ref{paralemlocal} and \ref{paralemboundary}, we immediately obtain the conclusion.
\end{proof}

\section{Existence results}

In this section, we present an existence result of $L^p$-viscosity solutions of (\ref{E0}) under suitable conditions when obstacles are merely continuous. 

Using the parabolic mollifier by $\rho\in C^\infty_0(\R^{n+1})$ with $\rho\geq0$ in $\R^{n+1}$, $\rho\equiv 0$ in $\R^{n+1}\setminus Q_1$ and 
$\iint_{R^{n+1}} \rho \, dxdt=1$, we introduce smooth approximations of $f$, $\mu$ and $F$ by 
\[
f_\e:=f\ast\rho_\e,\quad \mu_\e:=\mu\ast\rho_\e\]
and
\[F_\e(x, t, \xi, X):=\iint_{\R^{n+1}}\rho_\e(x-y, t-s)F(y, s, \xi, X)\, dyds
\]
for $(x, t, \xi, X)\in \R^{n+1}\ti\R^n\ti S^n$, where $\rho_\e(x, t):=\e^{-n-2}\rho(\fr x\e, \fr{t}{\e^2})$. Here and later, we use the same notion $f$, $\mu$ and $F$ for their zero extension outside of $\O_T$. Under (\ref{paraAf}), (\ref{paraAmu}) and (\ref{paraAF}), it is easy to observe that for $(x, t, \xi, X)\in\R^{n+1}\ti\R^n\ti S^n$, 
\begin{equation}\label{Approxprop}
\le\{
\begin{array}{ll}
(i) &\P^-(X)-\mu_\e(x, t)|\xi|\leq F_\e(x, t, \xi, X)\leq \P^+(X)+\mu_\e(x, t)|\xi|,\\
(ii)& \|f_\e\|_{L^p(\R^{n+1})}\leq \|f\|_{L^p(\O_T)},\\
(iii)& \|\mu_\e\|_{L^q(\R^{n+1})}\leq \|\mu\|_{L^q(\O_T)}. 
\end{array}\ri.
\end{equation}

Furthermore, we shall suppose that $\varphi$ and $\psi$ are defined in a neighborhood of $\ol \O_T$ with the same modulus of continuity. More precisely, there is $\e_1>0$ such that for $(x, t), (y, s)\in N_{\e_1}$, 
\begin{equation}\label{paraAphipsi}
|\varphi(x, t)-\varphi(y, s)|\vee|\psi(x, t)-\psi(y, s)|\leq \sigma_0(d((x, t), (y, s))), 
\end{equation}
where $N_{\e_1}:=\{(x, t)\in \R^{n+1} : \dist((x, t), \O_T)<\e_1\}$. Under (\ref{paraAphipsi}), we define $\varphi_\e$ and $\psi_\e$ as follows: \[
\varphi_\e:=\varphi\ast\eta_\e-\sigma_0(\sqrt{2}\e), \quad\mbox{and}\quad \psi_\e:=\psi\ast\eta_\e+\sigma_0(\sqrt{2}\e).
\]
It is easy to see that for $\e\in(0, \fr{\e_1}{\sqrt{2}})$, \[
\varphi_\e\leq g\leq \psi_\e\quad\mbox{on }\p_p \O_T, 
\]
and 
\[
|\varphi_\e(x, t)-\varphi_\e(y, s)|\vee|\psi_\e(x, t)-\psi_\e(y, s)|\leq \sigma_0(d((x, t), (y, s)))
\]
for $(x, t), (y, s)\in \ol \O_T$. 

For $\e>0$ and $\delta>0$, we shall consider approximate equations: 
\begin{equation}\label{paraApproximate}
u_t+F_\e(x, t, Du, D^2u)+\fr{1}{\delta}(u-\psi_\e)^+-\fr{1}{\delta}(\varphi_\e-u)^+=f_\e\quad\mbox{in }\O_T.
\end{equation}

In order to apply an existence result in \cite{CKLS}, we shall suppose the uniform exterior cone condition on $\O$ in \cite{M}, which is stronger than (\ref{paraAboundary}).
\begin{prop}\label{prop:paraEx}
Under $(\ref{paraApq})$, $(\ref{paraAf})$, $(\ref{paraAmu})$, $(\ref{paraAF})$, $(\ref{paraAob})$, $(\ref{paraAconti})$ and $(\ref{paraAphipsi})$, we assume the uniform exterior cone condition on $\O$. Then, there exists a $C$-viscosity solution $u_\e^\delta\in C(\ol \O_T)$ of $(\ref{paraApproximate})$ satisfying (\ref{paraBC}).
\end{prop}

We next show an existence result for (\ref{E0}) when $\varphi$, $\psi$, $F$ and $f$ are smooth. 

\begin{thm}(cf. Theorem 2.1 in \cite{CKLS})
\label{thm:approx}
Under the same hypotheses in Proposition $\ref{prop:paraEx}$, let $u_\delta^\e\in C(\ol \O_T)$ be $C$-viscosity solutions of $(\ref{paraApproximate})$ satisfying $(\ref{paraBC})$. For each small $\e>0$, there exist $\delta_\e>0$ and $\hat C_\e>0$ such that 
\begin{equation}\label{paraEstBeta}
0\leq \fr{1}{\delta}(u_\e^\delta-\psi_\e)^++\fr{1}{\delta}(\varphi_\e-u^\delta_\e)^+\leq \hat C_\e\quad \mbox{in }\ol \O_T\quad\mbox{for } \delta\in(0, \delta_\e).  
\end{equation}
Furthermore, there exist a subsequence $\{\delta_k\}_{k=1}^\infty$ and $u_\e\in C(\ol \O_T)$ such that $\delta_k\to0$ as $k\to\infty$, $(\ref{paraBC})$ holds for $u_\e$, 
\begin{equation}\label{paraConverge}
u_\e^{\delta_k}\to u_\e \quad\mbox{uniformly in }\ol \O_T, \mbox{ as } k\to\infty,
\end{equation}
and $u_\e$ is a (unique) $C$-viscosity solution of 
 \begin{equation}\label{ApproximateEq}
\min\{\max\{u_t+F_\e(x, t, Du, D^2u)-f_\e, u-\psi_\e\}, u-\varphi_\e\}=0\quad\mbox{in }\O_T. 
  \end{equation}
\end{thm}

\begin{proof}
To prove the estimate on $\fr{1}{\delta}(\varphi_\e-u_\e^\delta)^+$, independent of $\delta$, we let $(x_0, t_0)\in \ol \O_T$ satisfy that 
\[
\max_{\ol \O_T}\fr{1}{\delta}(\varphi_\e-u_\e^\delta)^+=\fr{1}{\delta}(\varphi_\e-u_\e^\delta)(x_0, t_0)^+>0.\]
Thus, we observe that $u_\e^\delta-\varphi_\e$ attains its minimum at $(x_0, t_0)\in \O_T$. Hence, the definition implies 
\[
0\leq \fr{1}{\delta}(\varphi_\e-u_\e^\delta)^+\leq (\varphi_\e)_t+F_\e(x_0, t_0, D\varphi_\e, D^2\varphi_\e)-f_\e\quad \mbox{at } (x_0, t_0)
\]
because of $\fr{1}{\delta}(u_\e^\delta-\psi_\e)(x_0, t_0)^+=0$.

It follows from the same argument that the estimate on $\fr{1}{\delta}(u_\e^\delta-\psi_\e)^+$ holds. Thus, we conclude the first assertion (\ref{paraEstBeta}). This implies the $L^\infty$ bound of $u_\e^\delta$ independent of $\delta\in(0, 1)$ for each $\e\in(0, 1)$. 

By regarding the penalty term as the right hand side, it is standard to establish the equi-continuity and uniform boundedness of $\{u_\e^\delta\}_{\delta>0}$ for each $\e>0$. Therefore, by Ascoli-Arzel\`a theorem, we can find a subsequence $\{u_\e^{\delta_k}\}_{k=1}^\infty$ and $u_\e\in C(\ol \O_T)$ satisfying (\ref{paraConverge}). 

We shall show that $u_\e$ is a $C$-viscosity supersolution of (\ref{ApproximateEq}) by contradiction. Thus, we suppose that $u_\e-\eta$ attains its local strict minimum at $(x_0, t_0)\in \O_T$ for $\eta\in C^{2, 1}(\O_T)$, and 
\begin{equation}\label{Verify}
\min\{\max\{\eta_t+F_\e(x_0, t_0, D\eta, D^2\eta)-f_\e, u_\e-\psi_\e\}, u_\e-\varphi_\e\}\leq -2\theta\mbox{ at }(x_0, t_0)
\end{equation}
for some $\theta>0$. By the uniform convergence, we may suppose that $u_\e^{\delta_k}-\eta$ attains its local minimum at $(x_{\delta_k}, t_{\delta_k})\in \O_T$, where $(x_{\delta_k}, t_{\delta_k})\to (x_0, t_0)$ as $k\to\infty$. For simplicity, we shall write $\delta$ for $\delta_k$. 

By (\ref{Verify}), since we may suppose that  
\[
(u_\e^\delta-\psi_\e)(x_\delta, t_\delta)\leq -\theta\quad\mbox{for small } \delta>0, 
\]
we have $\fr{1}{\delta}(u_\e^\delta-\psi_\e)^+=0$ at $(x_\delta, t_\delta)$. Hence, sending $k\to\infty$ in (\ref{paraApproximate}) with $\delta=\delta_k$, we obtain
\[
\eta_t+F_\e(x_0, t_0, D\eta, D^2\eta)\geq f_\e\quad\mbox{at }(x_0, t_0),
\]
which together with (\ref{Verify}) yields 
\[
(u_\e^\delta-\varphi_\e)(x_\delta, t_\delta)\leq -\theta
\]
for small $\delta>0$. However, this together with (\ref{paraEstBeta}) yields a contradiction for large $k\geq 1$. 
\end{proof}

Now, we shall show our proof of Theorem \ref{thm:paraEx}. 

\begin{proof}[Proof of Theorem \ref{thm:paraEx}]
Let $u_\e\in C(\ol \O_T)$ be $C$-viscosity solutions of (\ref{ApproximateEq}) satisfying (\ref{paraBC}) in Theorem \ref{thm:approx}. In view of Lemma 2.9 in \cite{CKS}, since $F_\e$ and $f_\e$ are continuous, it is known to see that $u_\e$ is an $L^p$-viscosity solution of (\ref{ApproximateEq}). Furthermore, recalling (\ref{Approxprop}), from Theorem \ref{cor:Global}, there is a modulus of continuity $\o_0$, independent of $\e$, such that 
\[
|u_\e(x, t)-u_\e(y, s)|\leq \o_0(d((x, t), (y, s)))\quad\mbox{for }(x, t), (y, s)\in \ol \O_T. 
\]
This together with Proposition \ref{prop:bounded} implies that 
there are a subsequence $\e_k>0$ and $u\in C(\ol \O_T)$ such that $\e_k\to 0$ as $k\to\infty$, and $u_{\e_k}$ converges to $u$ uniformly in $\ol \O_T$. In what follows, we shall write $\e$ for $\e_k$. 

It remains to show that $u$ is an $L^p$-viscosity solution of (\ref{E0}). Suppose by contradiction that for some $\eta\in W^{2, 1}_p(\O_T)$, $u-\eta$ attains its local strict minimum at $(x_0, t_0)\in \O_T$, and 
\[
\min\{\max\{\eta_t+F(x, t, D\eta, D^2\eta)-f, u-\psi\}, u-\varphi\}\leq -2\theta
\]
a.e. in $Q_{2r}(x_0, t_0)\Subset \O_T$ for some $\theta$, $r>0$. For the sake of simplicity, we may suppose that $(x_0, t_0)=(0, 0)$. 
Since we may suppose that for small $\e>0$,
\[
(u_\e-\psi_\e)\leq -\theta \quad \mbox{in }Q_r,
\]
it suffices to consider the case when $u_\e$ is an $L^p$-viscosity supersolution of 
\begin{equation}\label{Subexist}
\min\le\{u_t+F_\e(x, t, Du, D^2u)-f_\e, u-\varphi_\e\ri\}=0\quad\mbox{in }Q_r.
\end{equation}
Thus, Proposition \ref{prop:bounded} implies that
\[
u\geq \varphi\quad\mbox{in }Q_r.
\]
Hence, $\eta\in W^{2, 1}_{p}(Q_{r})$ satisfies that 
\begin{equation}\label{eq:conclude}
\eta_t+F(x, t, D\eta, D^2\eta)\leq f-\theta\quad\mbox{a.e. in }Q_r.
\end{equation}

On the other hand, following the argument in the proof of Theorem 4.1 in \cite{CCKS}, since $u_\e$ is an $L^p$-viscosity supersolution of (\ref{Subexist}) together with the uniform convergence of $u_\e$ to $u$, we see that $u$ is an $L^p$-viscosity supersolution of
\[
u_t+F(x, t, Du, D^2u)-f=0\quad\mbox{in }Q_r,
\]
which contradicts (\ref{eq:conclude}). In order to apply the stability result in \cite{CKS}, we only note that $\mu D\eta\in L^p(Q_r)$ holds true since $q>n+2$, and $\eta\in W^{2, 1}_{p}(Q_r)$ for $q\geq p$ though $\mu$ may not be in $L^\infty$ in (\ref{paraAF}).
\end{proof}

\section{Local H\"older continuity of the space derivative}

\subsection{Estimates in the non-coincidence set}

We first note that $L^p$-viscosity solutions $u\in C(\O_T)$ of $(\ref{E0})$ are also $L^p$-viscosity solutions of 
\begin{equation*}
u_t+F(x, t, Du, D^2u)-f=0\quad\mbox{in }N[u].	
\end{equation*}

For any compact set $K\Subset N[u]$, where $u\in C(\O_T)$ is an $L^p$-viscosity solution of $(\ref{E0})$, we first show that $Du\in C^{\beta}(K)$ for some $\beta\in(0, 1)$. 
\begin{prop}\label{CKS}(cf. Theorem 7.3 in \cite{CKS})
Assume $(\ref{paraApq2})$, $(\ref{paraAf})$, $(\ref{paraAmu})$, $(\ref{paraAF})$. 
Let $\beta\in(0, \hat\beta\wedge\beta_0)$ be fixed, where $\hat\beta$ is from Proposition $\ref{prop:C1}$, and $\beta_0=1-\fr{n+2}{p}$. 
Then, there are $\delta_0>0$ and $r_1>0$, depending on $n$, $\L$, $\l$, $p$, $q$ and $\beta$, such that if $u\in C(\O_T)$ is an $L^p$-viscosity solution of $(\ref{E0})$, and if $(\ref{paraAF2})$ holds for $\e=r_1$, then $u\in C^{1, \beta}(N_{r_1}[u])$. Moreover, if $Q_{2R}(x, t)\subset N_{r_1}[u]$, then there is $C>0$, depending on $n$, $\L$, $\l$, $p$, $q$, $\beta$, $\delta_0$, $r_1$, $\|\mu\|_{L^q(\O_T)}$, such that \[
\fr{|Du(y, s)-Du(z, \tau)|}{d((y, s), (z, \tau))^{\beta}}\leq \fr{C}{R^{1+\beta}}\le(\|u\|_{L^\infty(\O_T)}+R^{1+\beta_0}\|f\|_{L^p(\O_T)}\ri)\] 
for $(y, s), (z, \tau)\in Q_{R}(x, t)$ with $(y, s)\neq(z, \tau)$.  
\end{prop}

Before going to the proof of Proposition \ref{CKS}, we first show a lemma corresponding to Lemma 6.3 in \cite{CKS}.
    
In the next lemma, for a modulus of continuity $\sigma$ and a constant $\k>0$,
we use the space
\[
C(\sigma, \k; \ol{Q}_1):=\le\{h\in C(\ol{Q}_1)
\le.: 
\begin{array}{ll}
|h(x, t)-h(y, s)|\leq \sigma(d((x, t), (y, s))) \mbox{ for} \\
(x, t), (y, s)\in \p_p Q_1, \mbox{ and } \|h\|_{L^\infty(Q_1)}\leq \k
\end{array}\ri. \ri\}.
\]
For $\zeta^\ast\in \R^n$, which will be fixed in the proof of Proposition \ref{CKS}, we also introduce
\[
G(x, t, \xi, X):=F(x, t, \xi+\zeta^\ast, X)-F(x, t, \xi, X). 
\]
Note that
\[
g^\ast(x, t):=\sup\le\{|G(x, t, \xi, X)| : \xi\in\R^n, X\in S^n\ri\}\leq |\zeta^\ast|\mu(x, t).
\]

\begin{lem}\label{prop:approximate}
(cf. Lemma 6.3 in \cite{CKS})
Assume $(\ref{paraApq2})$. Let $\sigma$ be a modulus of continuity, and let $\k>0$ and $p'\in(n+2, p)$ be constants. Then, for any $\e>0$, there exists $\delta_1=\delta_1(\e, p', n, \L, \l, p, q, \sigma , \k)\in(0, 1)$ such that if $f, \mu$ and $F$ in $(\ref{paraAf})$, $(\ref{paraAmu})$ and $(\ref{paraAF})$ for $\O_T=Q_1$, respectively, satisfy   
\begin{equation}\label{eq:smallness}
\|f\|_{L^{n+2}(Q_1)}\vee\|\mu\|_{L^{p'}(Q_1)}\vee\|g^\ast\|_{L^{n+2}(Q_1)}\vee \|\theta((0, 0), \cdot)\|_{L^{n+2}(Q_1)}\leq \delta_1,
\end{equation}
then for any two $L^p$-viscosity solutions $v$ and $h\in C(\sigma, \k; \ol{Q}_1)$ of 
\[
v_t+F(x, t, Dv, D^2v)+G(x, t, Dv, D^2v)-f=0\quad\mbox{in }Q_1
\]
and
\begin{equation*}\label{eq:limit}
h_t+F(0, 0, 0, D^2h)=0\quad\mbox{in }Q_1,
\end{equation*}
respectively, satisfying $(v-h)|_{\p_p Q_1}=0$, it follows that 
\[
\|v-h\|_{L^\infty(Q_1)}\leq \e. 
\]
\end{lem}

\begin{rem}
We notice that $\|\mu\|_{L^{p'}(Q_1)}\leq \d_1$ in $(\ref{eq:smallness})$ for $p'\in(n+2, p)$ because we do not know if the equi-continuity of $v_k$ holds true in the proof below when $p'=n+2$. 
\end{rem}

\begin{proof}
We argue by contradiction. Suppose that there are $\hat{\e}>0$, $\zeta_k^\ast\in \R^n$, $v_k$, $h_k\in C(\sigma, \k; \ol{Q}_1)$, $f_k\in L^p(Q_1)$, $\mu_k\in L^q(Q_1)$ and $F_k:Q_1\ti\R^n\ti S^n\to\R$ satisfying that 
\begin{equation}\label{Aunif}
\begin{cases}
&F_k(x, t, 0, O)=0, \mbox{ and}\\
&\P^-(X-Y)-\mu_k(x, t)|\xi-\zeta|\leq F_k(x, t, \xi, X)-F_k(x, t, \zeta, Y)\\
&\leq \P^+(X-Y)+\mu_k(x, t)|\xi-\zeta|
\end{cases}
\end{equation}
for $(x, t)\in Q_1$, $\xi$, $\zeta\in\R^n$, $X$, $Y\in S^n$ such that 
\begin{equation}\label{A:approx}
\|f_k\|_{L^{n+2}(Q_1)}\vee\|\mu_k\|_{L^{p'}(Q_1)}\vee\|g_{k}^\ast\|_{L^{n+2}(Q_1)}\vee \|\theta_{k}((0, 0), \cdot)\|_{L^{n+2}(Q_1)}\leq \fr{1}{k},
\end{equation}
where $g_{k}^\ast(x, t):=\sup\le\{|F_k(x, t, \xi+\zeta_{k}^\ast, X)-F_k(x, t, \xi, X)| : \xi\in\R^n, X\in  S^n\ri\}$, and
\[\theta_{k}((x, t), (y, s)):=\sup_{X\in S^n}\fr{|F_k(x, t, 0, X)-F_k(y, s, 0, X)|}{1+\|X\|}.\]
Furthermore, by setting 
$G_k(x, t, \xi, X):=F_k(x, t, \xi+\zeta_{k}^\ast, X)-F_k(x, t, \xi, X),$
we suppose that $v_k$ and $h_k$ are, respectively, $L^p$-viscosity solutions of 
\[
(v_k)_t+F_k(x, t, Dv_k, D^2v_k)+G_k(x, t, Dv_k, D^2v_k)-f_k=0\quad\mbox{in }Q_1
\]
and
\[
(h_k)_t+F_k(0, 0, 0, D^2h_k)=0\quad\mbox{in }Q_1,
\]
which satisfy that $(v_k-h_k)|_{\p_p Q_1}$=0, and 
\begin{equation}\label{contradiction}
\|v_k-h_k\|_{L^\infty(Q_1)}\geq\hat{\e}.
\end{equation}

In view of Ascoli-Arzel\`a theorem, we can find $v$, $h\in C(\sigma, \k; \ol{Q}_1)$ such that $v_k\to v$ and $h_k\to h$ in $C(\ol{Q}_1)$ as $k\to\infty$, and $v=h$ on $\p_p Q_1$. By (\ref{Aunif}), we may suppose that $F_k(0, 0, 0, X)$ converges $F_\infty(X)$ uniformly in any compact sets in $S^n$, which satisfies that
\begin{equation*}\label{Ainfty}
F_\infty(O)=0, \quad\mbox{and}\quad \P^-(X-Y)\leq F_\infty(X)-F_\infty(Y)\leq \P^+(X-Y).
\end{equation*}
We also notice that by (\ref{Aunif}) and (\ref{A:approx}), for $\eta\in W^{2, 1}_{n+2}(Q_1)$ and $Q_r(x_0, t_0)\Subset Q_1$, we have
\[
\|F_k(\cdot, D\eta(\cdot), D^2\eta(\cdot))-F_\infty(D^2\eta(\cdot))\|_{L^{n+2}(Q_r(x_0, t_0))}\to 0\quad\mbox{as }k\to\infty. 
\]
Hence, since $F_\infty$ is continuous, in view of Theorem 6.1 in \cite{CKS}, we verify that $v$ and $h$ are $L^{n+2}$-viscosity (thus, $C$-viscosity) solutions of 
\[
F_\infty(D^2u)=0\quad\mbox{in } Q_1.
\]
Therefore, the comparison principle implies that $v=h$ in $\ol{Q}_1$, which contradicts (\ref{contradiction}). 
\end{proof}

Although our proof of Proposition \ref{CKS} follows by the same argument as in \cite{CKS}, we give a proof because we need some modification.

\begin{proof}[Proof of Proposition \ref{CKS}]
Let $\beta\in(0, \hat\beta\wedge\beta_0)$ and $p'\in(n+2, p)$ be fixed constants, where $\hat\beta$ is from Proposition \ref{prop:C1}, and $\beta_0:=1-\fr{n+2}{p}$.  Without loss of generality, we can assume that $(x, t)=(0, 0)$ and $R=1$ hereafter. We choose 
\[
K_1=K_1\le(n, \L, \l, p, q, \|\mu\|_{L^q(Q_2)}\ri)\quad
\mbox{and}\quad 
\hat{\alpha}=\hat{\alpha}\le(n, \L, \l, p, q, \|\mu\|_{L^q(Q_2)}\ri)
\]
in Proposition \ref{prop:Holder} for $R=1$. 

For small $\tau\in(0, 1)$, which will be fixed later, setting 
\[
\e:=K_2\tau^{1+\hat{\beta}}, 
\] 
where $K_2=K_2(n, \L, \l)$ is the constant in Proposition \ref{prop:C1},
we choose $\delta_1=\delta_1(\e, p', n, \L, \l, p, q, \sigma)\in(0, 1)$ in Lemma \ref{prop:approximate} for $\k=1$, where the modulus of continuity $\sigma$ is given by 
\[
\sigma(r)=K_1 r^{\hat{\alpha}}.
\]  

Now, for $\rho\in(0, 1)$, which will be fixed later, we set $\hat{u}(x, t):=\fr{1}{N}u(\rho x, \rho^2 t),$
 where
\[N:=2\|u\|_{L^\infty(Q_2)}+\fr{2^{1+\beta}}{\delta_1}\sup_{0<r\leq 2} \fr{1}{r^{\beta}}\|f\|_{L^{n+2}(Q_r)}.\]
We notice that $N<\infty$.

For $\tau\in(0, \tau_0]$, where $\tau_0:=2^{-\fr{1}{\beta}}$, we prove by induction that there is a sequence of affine functions $\ell_k(x)=a_k+\la b_k, x \ra$ such that 
\[
v_k=v_k(x, t):=\tau^{-k(1+\beta)}\le(\hat{u}(\tau^k x, \tau^{2k}t)-\ell_k(\tau^k x)\ri)
\]
satisfies that 
\begin{equation}\label{inductive}
\begin{cases}
\mbox{(i)} &\|v_k\|_{L^\infty(Q_2)}\leq 1\\
\mbox{(ii)} &|a_k-a_{k-1}|\vee\le(\tau^{k-1}|b_k-b_{k-1}|\ri)\leq K_2\tau^{(k-1)(1+\beta)}\\
\mbox{(iii)} &|v_k(x, t)-v_k(y, s)|\leq K_1 d((x, t), (y, s))^{\hat\a}\mbox{ for }(x, t), (y, s)\in \ol Q_1,
\end{cases}
\end{equation}
where $\ell_{-1}=\ell_0\equiv 0$. 
 
 For $k=0$, since $v_0=\hat{u}$, by the definition of $N$, the inequality (i) holds for $k=0$ while (ii) is trivially satisfied for $k=0$. Since $\hat{u}$ is an $L^p$-viscosity subsolution and an $L^p$-viscosity supersolution, respectively, of 
 \[
u_t+\P^-(D^2u)-\hat{\mu}|Du|-\hat{f}=0\ \mbox{ and }\ u_t+\P^+(D^2u)+\hat{\mu}|Du|-\hat{f}=0\quad\mbox{in }Q_2,
\]
where $\hat{\mu}(x, t)=\rho\mu(\rho x, \rho^2 t)$ and $\hat{f}(x, t)=\fr{\rho^2}{N} f(\rho x, \rho^2 t)$, it follows from Proposition \ref{prop:Holder} that
\begin{align*}
|\hat u(x, t)-\hat u(y, s)|&\leq K_1d((x, t), (y, s))^{\hat \a}\le(\|\hat{u}\|_{L^\infty(Q_2)}+\fr{1}{N}\|f\|_{L^{n+2}(Q_2)}\ri)\\
&\leq K_1d((x, t), (y, s))^{\hat \a}
\end{align*}
for $(x, t)$, $(y, s)\in\ol Q_1$,
which is (iii) for $k=0$. Notice that the last inequality is derived because of our choice of $\d_1$ and $N$.

By induction, assume that (\ref{inductive}) holds for $k=j$. We observe that $v_j$ is an $L^p$-viscosity solution of  
\[
u_t+F_j(x, t, Du, D^2u)+G_j(x, t, Du, D^2u)-f_j=0\quad\mbox{in }Q_2,
\]
where 
\begin{align*}
F_j(x, t, \xi, X)&:=\fr{\rho^2 \tau^{j(1-\beta)}}{N}F\le(\rho \tau^j x, \rho^2 \tau^{2j}t, \fr{N\tau^{j\beta}}{\rho}\xi, \fr{N}{\rho^2\tau^{j(1-\beta)}}X\ri), \\
G_j(x, t, \xi, X)&:=F_j(x, t, \xi+\tau^{-j\beta}b_j, X)-F_j(x, t, \xi, X),\\
f_j(x, t)&:=\fr{\rho^2\tau^{j(1-\beta)}}{N}f(\rho \tau^j x, \rho^2 \tau^{2j}t). 
\end{align*}
We notice that $\zeta^\ast$ in $G$ in Lemma $\ref{prop:approximate}$ corresponds to $\tau^{-j\b}b_j$. 
We note that for $(x, t)\in Q_2$, $\xi$, $\zeta\in\R^n$, $X$, $Y\in S^n$, 
\begin{equation*}
\begin{cases}
&F_j(x, t, 0, O)=0, \mbox{ and}\\
&\P^-(X-Y)-\mu_j(x, t)|\xi-\zeta|\leq F_j(x, t, \xi, X)-F_j(x, t, \zeta, Y)\\
&\leq \P^+(X-Y)+\mu_j(x, t)|\xi-\zeta|,
\end{cases}
\end{equation*}
where $\mu_j(x, t):=\rho\tau^j\mu(\rho \tau^j x, \rho^2 \tau^{2j}t)$. Also, since 
\[
|G_j(x, t, \xi, X)|\leq \rho|b_j|\tau^{j(1-\beta)}\mu(\rho \tau^j x, \rho^2 \tau^{2j}t) \ \mbox{ for } (x, t, \xi, X)\in Q_2\ti\R^n\ti S^n,
\] 
setting $g_j^\ast(x, t):=\sup\{|G_j(x, t, \xi, X)| : \xi\in\R^n, X\in S^n\}$, we have
\begin{equation}\label{g_j}
\|g_{j}^\ast\|_{L^{n+2}(Q_2)}
\leq \fr{|b_j|}{\tau^{j\beta}} \|\mu\|_{L^{n+2}(Q_{2\rho \tau^j})}
\leq 2^{\beta_0}\o_n\rho^{\beta_0}\tau^{j(\beta_0-\beta)}|b_j|\|\mu\|_{L^p(Q_{2\rho\tau^j})},
\end{equation}
where $\o_n:=|B_1|^{\fr{1}{n+2}-\fr{1}{p}}$. By induction, we have 
\begin{equation}\label{b_k}
|b_j|\leq K_2 \sum_{k=0}^{j-1} \tau^{k\beta}\leq \fr{K_2}{1-\tau^\beta}\leq 2K_2
\end{equation}
because $0<\tau\leq \tau_0=2^{-\fr{1}{\beta}}$. Simple calculations together with our choice of $N$ and (\ref{g_j}) give
\begin{equation*}\label{3case}
\begin{cases}
(1) &\|f_j\|_{L^{n+2}(Q_2)}\leq \fr{\d_1}{2},\\
(2) &\|\mu_j\|_{L^{p'}(Q_2)}\leq (\rho\tau^j)^{1-\fr{n+2}{p'}}\|\mu\|_{L^{p'}(Q_{2\rho\tau^j})}\leq \rho^{1-\fr{n+2}{p'}}\|\mu\|_{L^{p'}(Q_{2\rho\tau^j})},\\
(3) &\|g_{j}^\ast\|_{L^{n+2}(Q_2)}
\leq 2^{\beta_0+1}K_2\o_n\rho^{\beta_0}\|\mu\|_{L^p(Q_{2\rho\tau^j})}.
\end{cases}
\end{equation*}
Now, we can choose $\rho\in(0, 1)$, independent of $j\geq0$, such that 
\[
\|\mu_j\|_{L^{p'}(Q_1)}\vee\|g_{j}^\ast\|_{L^{n+2}(Q_1)}\leq \delta_1, \quad\mbox{and}\quad \rho\leq \sqrt{N}.  
\]
Because $\rho\leq \sqrt{N}$, assuming (\ref{paraAF2}) for $\delta_0=\delta_1>0$,    
we have
\[
\|\theta_{j}((0, 0), \cdot)\|_{L^{n+2}(Q_1)}\leq \fr{1}{\rho\tau^j}\|\theta((0, 0), \cdot)\|_{L^{n+2}(Q_{\rho\tau^j})}\leq \delta_1,
\]
where $\theta_{j}((x, t), (y, s)):=\sup_{X\in S^n}|F_j(x, t, 0, X)-F_j(y, s, 0, X)|/(1+\|X\|)$. 

Let $h\in C(\ol{Q}_1)$ be a $C$-viscosity solution of 
\[
u_t+F_j(0, 0, 0, D^2u)=0\quad\mbox{in }Q_1
\]
satisfying $h=v_j$ on $\p_p Q_1$. Hence, in view of Lemma \ref{prop:approximate} and Proposition \ref{prop:C1}, we have 
\begin{equation}\label{eq:v}
\|v_j-h\|_{L^\infty(Q_1)}\leq \e=K_2 \tau^{1+\hat\beta}\quad\mbox{and}\quad \|h\|_{C^{1, \hat\beta}(\ol{Q}_\fr12)}\leq K_2\|h\|_{L^\infty(Q_1)}\leq K_2. 
\end{equation} 

We define 
\[
\ell_{j+1}(x):=\ell_j(x)+\tau^{j(1+\beta)}\le(h(0, 0)+\le\la Dh(0, 0), \fr{x}{\tau^{j}}\ri\ra\ri). 
\]
Since the definition of $v_{j+1}$ can be written by 
\begin{align*}
&v_{j+1}(x, t)\\
&=\tau^{-(j+1)(1+\beta)}\le(\hat{u}(\tau^{j+1}x, \tau^{2(j+1)}t)-\ell_{j+1}(\tau^{j+1}x)\ri)\\
&=\tau^{-(j+1)(1+\beta)}\le\{\hat{u}(\tau^{j+1}x, \tau^{2(j+1)}t)-\ell_j(\tau^{j+1}x)-\tau^{j(1+\beta)}\le(h(0, 0)\ri.\ri.\\
&\quad\le.\le.+\la Dh(0, 0), \tau x\ra\ri)
\ri\}\\
&=\tau^{-(1+\beta)}\le(v_j(\tau x, \tau^2 t)-h(0, 0)-\la Dh(0, 0), \tau x\ra\ri),
\end{align*} 

it follows from (\ref{eq:v}) that for $\tau\in(0, \tau_1]$, where $\tau_1:=\tau_0\wedge  (2K_2 2^{\fr{3}{2}(1+\hat\beta)})^{-\fr{1}{\hat\beta-\beta}}$, 
\begin{align}\label{eq:k+1}
\|v_{j+1}\|_{L^\infty(Q_2)}&\leq \tau^{-(1+\beta)}\le\{\|v_j-h\|_{L^\infty(Q_{2\tau})}+K_2\le(\sqrt{(2\tau)^2+4\tau^2}\ri)^{1+\hat{\beta}}\ri\}\nonumber\\
&\leq \tau^{-(1+\beta)}\le( K_2 \tau^{1+\hat{\beta}}+K_22^{\fr{3}{2}(1+\hat{\beta})}\tau^{1+\hat{\beta}}\ri)\nonumber\\
&\leq 2K_2 2^{\fr{3}{2}(1+\hat{\beta})}\tau^{\hat{\beta}-\beta},\end{align}
which yields (i) of (\ref{inductive}) for $k=j+1$. 
Noting that $a_{j+1}-a_j=\tau^{j(1+\beta)}h(0, 0)$, $b_{j+1}-b_j=\tau^{j\beta}Dh(0, 0)$ and the second inequality of (\ref{eq:v}), we obtain (ii) of (\ref{inductive}) for $k=j+1$. 

It remains to show (iii) of (\ref{inductive}) for $k=j+1$. 
Since $v_{j+1}$ is an $L^p$-viscosity subsolution and an $L^p$-viscosity supersolution, respectively of   
\[
u_t+\P^-(D^2u)-\mu_{j+1}|Du|-f_{j+1}-g_{j+1}^\ast=0 \quad\mbox{in }Q_2\]
and
\[
u_t+\P^+(D^2u)+\mu_{j+1}|Du|-f_{j+1}+g_{j+1}^\ast=0\quad\mbox{in }Q_2,
\]
 by Proposition \ref{prop:Holder} and (\ref{eq:k+1}), we have 
 \begin{align}\label{v_k+1}
\fr{|v_{j+1}(x, t)-v_{j+1}(y, s)|}{d((x, t), (y, s))^{\hat\a}}
\leq K_1\le(2K_2 2^{\fr{3}{2}(1+\hat\beta)}\tau^{\hat\beta-\beta}+\fr{\d_1}{2}+\|g_{j+1}^\ast\|_{L^{n+2}(Q_2)}\ri)
\end{align}
for $(x, t)$, $(y, s)\in \ol Q_1$ with $(x, t)\neq (y, s)$. By the same manner as (\ref{g_j}) and (\ref{b_k}), we have 
\begin{align*}
\|g_{j+1}^\ast\|_{L^{n+2}(Q_2)}&\leq 2^{\b_0}\o_n\rho^{\beta_0}\tau^{(j+1)(\beta_0-\beta)}|b_{j+1}|\|\mu\|_{L^p(Q_{2\rho\tau^{j+1}})}\\
&\leq 2^{1+\b_0}\o_n\tau^{\beta_0-\beta}K_2\|\mu\|_{L^p(Q_2)}.
\end{align*}
Hence, combining this inequality with (\ref{v_k+1}), we can choose smaller $\tau>0$, if necessary, to obtain (iii) of (\ref{inductive}) for $k=j+1$. 

By (ii) of (\ref{inductive}), we find $a_\infty\in\R$ and $b_\infty\in\R^{n}$ such that $(a_k, b_k)\to(a_\infty, b_\infty)$ as $k\to\infty$. For any $(x, t)\in Q_1$, we choose $k\in\N\cup\{0\}$ such that 
\[
(x, t)\in Q_{\tau^k}\setminus Q_{\tau^{k+1}}. 
\]
Since $(x, t)\not\in Q_{\tau^{k+1}}$, (i) of (\ref{inductive}) implies
\[
|\hat{u}(x, t)-a_k-\la b_k, x\ra|\leq \tau^{k(1+\beta)}\leq\tau^{-1-\beta}\le(|x|^2+|t|\ri)^{\fr{1+\beta}{2}}. 
\]
By sending $k\to\infty$ in the above, it follows 
\[
|\hat{u}(x, t)-a_\infty-\la b_\infty, x\ra|\leq\tau^{-1-\beta}\le(|x|^2+|t|\ri)^{\fr{1+\beta}{2}}. 
\]
Therefore, we obtain the local H\"older continuity of $Du$ with exponent $\beta$ by Lemma A.1 in \cite{AP}.
\end{proof}

\subsection{Estimates near the coincidence set}

Following the idea in \cite{Sh} (see also \cite{PS}, \cite{D}), we next prove that the space derivative of $L^p$-viscosity solutions $u$ of (\ref{E0}) is H\"older continuous with the H\"older exponent $\beta_2:=\beta_0\wedge \beta_1$ at coincidence points, where $u$ touches one of the obstacles.
We remark that we cannot apply the argument for elliptic equations (cf. \cite{KT}) because $Q_r(0, -3r^2)$ and $Q_r$ in (\ref{pppWHI}) are disjoint. 

 In what follows, we use the notation
\[
Q_r^+:=B_r\ti(-r^2, r^2]\quad\mbox{and}\quad Q_r^+(x, t):=(x, t)+Q_r^+. 
\]


\begin{lem}\label{leminductive}
Assume $(\ref{paraApq2})$, $(\ref{paraAf})$, $(\ref{paraAmu})$, $(\ref{paraAF})$ and $(\ref{paraAregob})$. Then, for small $\e>0$, there exists $\hat C_0=\hat C_0(\e)>0$ such that if $u\in C(\O_T)$ is an $L^p$-viscosity solution of $(\ref{E0})$ in $\O_T$, and $(x_0, t_0)\in C^+[u]\cap \O^\e_T$ (resp., $(x_0, t_0)\in C^-[u]\cap \O^\e_T$), then it follows that  
\begin{equation*}\label{result1}
|u(x, t)-\psi(x_0, t_0)-\langle D\psi(x_0, t_0), x-x_0\rangle|\leq \hat C_0r^{1+\beta_2}
\end{equation*}
\[
\le(\mbox{resp., } |u(x, t)-\varphi(x_0, t_0)-\langle D\varphi(x_0, t_0), x-x_0\rangle|\leq \hat C_0r^{1+\beta_2}\ri)
\]
for $(x, t)\in Q_{r}^+(x_0, t_0)\cap \O_T$.
In particular, $u$ has a space derivative at $(x_0, t_0)$, and 
\[
Du(x_0, t_0)=D\psi(x_0, t_0) \quad \le(\mbox{resp., }Du(x_0, t_0)=D\varphi(x_0, t_0)\ri).
\]
\end{lem}

Before going to the proof of Lemma \ref{leminductive}, we first show a lemma corresponding to Lemma 2.1 in \cite{Sh}. 

In the next lemma, for $\zeta^\ast\in \R^n$ and $R^\ast\geq 0$, which will be fixed in the proof of Lemma \ref{leminductive}, we introduce
\[
G(x, t, \xi, X):=F(x, t, \xi+\zeta^\ast, X)-F(x, t, \xi, X),
\]
\[
Q_r^\ast:=B_r\ti \le(-r^2, r^2\wedge (R^\ast r^2)\ri]\quad\mbox{and}\quad Q_r^\ast(x, t):=(x, t)+Q_r^\ast.
\]
We notice that $Q_r^\ast=Q_r^+$ if $R^\ast\geq 1$, and $Q_r^\ast=Q_r$ if $R^\ast=0$. 
We also note that
\[
g^\ast(x, t):=\sup_{(\xi, X)\in\R^n\ti S^n}|G(x, t, \xi, X)|\leq |\zeta^\ast|\mu(x, t) \quad\mbox{and}\quad Q_r^\ast\subset Q_r^+.
\]

\begin{lem}\label{lemflat}
Under $(\ref{paraApq2})$, we assume that
\[\Phi, \Psi\in C^{1, \beta_1}(Q_1^\ast) \quad\mbox{for }\beta_1\in(0,1), \quad\mbox{and}\quad \Phi\leq\Psi \quad\mbox{in }Q_1^\ast.\]
Then, 
there exist $\delta_2\in(0, 1)$ and $\nu\in(0, 1)$, depending on $n$, $\L$, $\l$, $p$, $q$ and $\beta_1$ such that if $f, \mu$ and $F$ in $(\ref{paraAf})$, $(\ref{paraAmu})$ and $(\ref{paraAF})$ for $\Omega_T:=Q_1^\ast$, respectively, satisfy  
\begin{equation}\label{A2}
\|f\|_{L^p(Q_1^\ast)}\vee\|\mu\|_{L^q(Q_1^\ast)}\vee\|g^\ast\|_{L^p(Q_1^\ast)}\leq \delta_2\end{equation}
and
\begin{equation}\label{A3}
\|\Psi\|_{L^\infty(Q_r^\ast)}\vee\sup_{(x, t)\in Q_r^\ast} |\Phi(x, t)-\Phi(0, 0)-\langle D\Phi(0, 0), x\rangle|\leq \delta_2r^{1+\beta_1} 
\end{equation}
\[
\le(\mbox{resp., }\|\Phi\|_{L^\infty(Q_r^\ast)}\vee \sup_{(x, t)\in Q_r^\ast} |\Psi(x, t)-\Psi(0, 0)-\langle D\Psi(0, 0), x\rangle|\leq \delta_2 r^{1+\beta_1}\ri)
\]
for $r\in(0, 1]$, then, 
for any $L^p$-viscosity solution $u\in C(Q_1^\ast)$ of 
\begin{equation}\label{E1}
\min\{\max\{u_t+F(x, t, Du, D^2u)+G(x, t, Du, D^2u)-f, u-\Psi\}, u-\Phi\}=0
\end{equation}
in $Q_1^\ast$ satisfying 
\begin{equation*}
\inf_{Q_1^\ast} u\geq -1, \quad\mbox{and}\quad u(0, 0)=\Psi(0, 0)=0
\end{equation*}
\[
\le(\mbox{resp., }\sup_{Q_1^\ast} u\leq 1,\quad \mbox{and}\quad u(0, 0)=\Phi(0, 0)=0\ri), 
\]
 it follows that 

\[
\inf_{Q_\nu^\ast} u\geq -\nu^{1+\beta_1} \quad\le(\mbox{resp., }\sup_{Q_\nu^\ast} u\leq \nu^{1+\beta_1}\ri).\]
\end{lem}

\begin{rem}
We will choose $\Phi$ and $\Psi$ in (\ref{5.3phi}) and (\ref{5.3psi}), respectively.	
\end{rem}

\begin{proof}
We only show a proof when $u(0, 0)=\Psi(0, 0)$ because the other one can be shown similarly.  

We argue by contradiction. Thus, suppose that there are $\zeta_{k}^\ast\in\R^n$, $R_{k}^\ast>0$, $\Phi_k$, $\Psi_k\in C^{1, \beta_1}(Q_{1,k}^\ast)$ with $\Phi_k\leq \Psi_k$ in $Q_{1,k}^\ast$, $u_k\in C(Q_{1,k}^\ast)$, $f_k\in L^p(Q_{1,k}^\ast)$, $\mu_k\in L^q(Q_{1,k}^\ast)$ and $F_k: Q_{1,k}^\ast\ti\R^n\ti S^n\to\R$ satisfying (\ref{paraAF}) with $\Omega_T=Q_{1,k}^\ast$; 
\begin{equation*}
\begin{cases}
&F_k(x, t, 0, O)=0, \mbox{ and}\\
&\P^-(X-Y)-\mu_k(x, t)|\xi-\zeta|\leq F_k(x, t, \xi, X)-F_k(x, t, \zeta, Y)\\
&\leq \P^+(X-Y)+\mu_k(x, t)|\xi-\zeta|
\end{cases}
\end{equation*}
for $(x, t)\in Q_{1,k}^\ast$, $\xi$, $\zeta\in\R^n$, $X$, $Y\in S^n$ (for $k\in\N$) such that  
\begin{equation}\label{eq:A1}
\|f_k\|_{L^p(Q_{1,k}^\ast)}\vee\|\mu_k\|_{L^q(Q_{1,k}^\ast)}\vee \|g_{k}^\ast\|_{L^p(Q_{1,k}^\ast)}\vee c_k\vee d_k\leq\fr{1}{k},
\end{equation}
\begin{equation}\label{eq:A2}
\inf_{Q_{1,k}^\ast} u_k\geq -1,
 \end{equation} 
 $u_k(0, 0)=\Psi_k(0, 0)=0$,
and
\begin{equation}\label{eq:A3}
\inf_{Q_{\nu,k}^\ast} u_k<-\nu^{1+\beta_1}\quad\mbox{for any }\nu\in(0, 1),
\end{equation}
where 
\begin{align*}
Q_{r, k}^\ast&:=B_r\ti\le(-r^2, r^2\wedge( R_{k}^\ast r^2)\ri]\quad\mbox{for }r\in(0, 1],\\
  g_{k}^\ast(x, t)&:=\sup\le\{|F_k(x, t, \xi+\zeta_{k}^\ast, X)-F_k(x, t, \xi, X)| : \xi\in\R^n, X\in S^n\ri\},\\
c_k&:=\sup_{r\in(0, 1]} \fr{1}{r^{1+\beta_1}}\|\Psi_k\|_{L^\infty(Q_{r, k}^\ast)},\\
d_k&:=\sup_{r\in(0, 1]}\sup_{(x, t)\in Q_{r, k}^\ast} \fr{1}{r^{1+\beta_1}}|\Phi_k(x, t)-\Phi_k(0, 0)-\langle D\Phi_k(0, 0), x\rangle|.  
\end{align*}
Moreover, by setting 
\[
 G_k(x, t, \xi, X):=F_k(x, t, \xi+\zeta_{k}^\ast, X)-F_k(x, t, \xi, X),\]
 $u_k$ are $L^p$-viscosity solutions of (\ref{E1}) in $Q_{1, k}^\ast$, where $F$, $G$, $f$, $\Phi$ and $\Psi$ are replaced, respectively, by $F_k$, $G_k$, $f_k$, $\Phi_k$ and $\Psi_k$. 
 
Since $u_k\leq \Psi_k$ in $Q_{1,k}^{\ast}$, it follows by (\ref{eq:A1}) and (\ref{eq:A2}) that
 \begin{equation}\label{eq:B1}
\|u_k\|_{L^\infty(Q_{1,k}^\ast)}\leq 1.
 \end{equation} 
Now, for $r\in(0, 1]$, 
 by Proposition \ref{prop:def}, we see that 
\[
u_k\vee\le(\Phi_k(0, 0)+\langle D\Phi_k(0, 0), x\rangle+d_k r^{1+\beta_1}\ri)\quad\mbox{and}\quad u_k\wedge (-c_k r^{1+\beta_1})\] are an $L^p$-viscosity subsolution and an $L^p$-viscosity supersolution, respectively, of 
\[
u_t+\P^-(D^2u)-\mu_k|Du|-g_{k}^\ast-f^+_k= 0 \quad\mbox{in }Q_{r,k}^\ast\]
and
\[u_t+\P^+(D^2u)+\mu_k|Du|+g_{k}^\ast+f^-_k= 0 \quad\mbox{in }Q_{r,k}^\ast.\]
 Hence, by the same argument as in Lemma \ref{paralemlocal}, recalling (\ref{eq:A1}) and (\ref{eq:B1}), we can find $\hat{\gamma}\in(0, 1)$ and $\tilde C_0>0$, independent of $k$, such that 
 \begin{equation*}
 \|u_k\|_{C^{\hat{\gamma}}(\ol{Q}_{\fr12,k}^\ast)}\leq \tilde C_0.
\end{equation*}

Because $1\wedge R_{k}^\ast\in [0, 1]$, we may suppose that 
\[
1\wedge R_{k}^\ast\to R_\infty^\ast\quad\mbox{as }k\to\infty\mbox{ for some } R_\infty^\ast\in[0, 1].
\]

$\ul{\mbox{Case I}: R_\infty^\ast =0.}$ 
Using Ascoli-Arzel\`a theorem, we may suppose that $u_k\to u_\infty$ in $C(\ol{Q}_{\fr12})$ as $k\to\infty$ for some $u_\infty\in C(\ol{Q}_{\fr12})$. By (\ref{eq:A1}), (\ref{eq:A2}) and (\ref{eq:A3}), sending $k\to\infty$, we may have 
\begin{equation}\label{eq:B2}
-1\leq u_\infty\leq 0\quad \mbox{in } \ol{Q}_{\fr12},\quad\mbox{and}\quad\inf_{Q_\nu} u_\infty\leq -\nu^{1+\beta_1}\quad\mbox{for }\nu\in\le(0, \fr12\ri].
\end{equation}
Next, setting
\[
\ol{u}_{k}:=u_k\vee\le(\Phi_k(0, 0)+\langle D\Phi_k(0, 0), x\rangle+\fr{d_k}{2^{1+\beta_1}}\ri),\mbox{ and }\ \ul{u}_{k}:=u_k\wedge \le(-\fr{c_k}{2^{1+\beta_1}}\ri),
\]
we claim that $\ol{u}_{k}$ and $\ul{u}_{k}$ converge to $u_\infty$ uniformly in $\ol{Q}_{\fr12}$.  
Indeed, since for $(x, t)\in \ol{Q}_{\fr12}$, 
\begin{align*}
0\leq (\ol{u}_{k}-u_k)(x, t) &= \le(\Phi_k(0, 0)+\langle D\Phi_k(0, 0), x\rangle+\fr{d_k}{2^{1+\beta_1}}-u_k(x, t)\ri)^+\\
&\leq \le(-\Phi_k(x, t)+\Phi_k(0, 0)+\langle D\Phi_k(0, 0), x\rangle+\fr{d_k}{2^{1+\beta_1}}\ri)^+\\
&\leq 2^{-\beta_1}d_k,
\end{align*}
it follows that 
\begin{align*}
\|\ol{u}_{k}-u_\infty\|_{L^\infty(\ol{Q}_{\fr 1 2})}&\leq \|\ol{u}_{k}-u_k\|_{L^\infty(\ol{Q}_{\fr 1 2})}+\|u_k-u_\infty\|_{L^\infty(\ol{Q}_{\fr 1 2})}\\
&\leq 2^{-\beta_1}d_k +\|u_k-u_\infty\|_{L^\infty(\ol{Q}_{\fr 1 2})}. 
\end{align*}
In contrast, since for $(x, t)\in \ol{Q}_{\fr12}$, 
\begin{align*}
0\leq (u_k-\ul{u}_{k})(x, t) &= \le(u_k(x, t)+\fr{c_k}{2^{1+\beta_1}}\ri)^+\leq \le(\Psi_k(x, t)+\fr{c_k}{2^{1+\beta_1}}\ri)^+\leq 2^{-\beta_1}c_k,
\end{align*}
we have
\begin{align*}
\|\ul{u}_{k}-u_\infty\|_{L^\infty(\ol{Q}_{\fr12})}\leq 2^{-\beta_1}c_k +\|u_k-u_\infty\|_{L^\infty(\ol{Q}_{\fr12})}.
\end{align*}
Hence, $\ol{u}_{k}\to u_\infty$, $\ul{u}_{k}\to u_\infty$ uniformly in $\ol{Q}_{\fr12}$ as $k\to\infty$.  
 
 Passing to the limit, by Theorem 6.1 in \cite{CKS}, we see that $u_\infty$ is a $C$-viscosity subsolution and a $C$-viscosity supersolution, respectively, of
\begin{equation}\label{eq:B21/2}
u_t+\P^-(D^2u)= 0\quad\mbox{in }Q_{\fr12}, \quad\mbox{and}\quad u_t+\P^+(D^2u)= 0\quad\mbox{in }Q_{\fr12}. 
\end{equation}
Because $-u_\infty\geq 0$ in $Q_{\fr12}$ and $-u_\infty(0, 0)=0$, the strong maximum principle yields \begin{equation}\label{eq:B3}
-u_\infty\equiv 0\quad\mbox{in } Q_{\fr12},
\end{equation}
which contradicts $(\ref{eq:B2})$.

 $\ul{\mbox{Case II}: R_\infty^\ast >0.}$ Set $Q':=B_{\fr12}\ti(-\fr14, \fr{R_\infty^\ast}{8}]$. Using Ascoli-Arzel\`a theorem, we may suppose that $u_k\to u_\infty$ in $C(\ol{Q}')$ as $k\to\infty$ for some $u_\infty\in C(\ol{Q}')$. By (\ref{eq:A1}), (\ref{eq:A2}) and (\ref{eq:A3}), sending $k\to\infty$, we may have 
\begin{equation}\label{eq:B2'}
-1\leq u_\infty\leq 0\ \mbox{ in } \ol{Q}',\quad\mbox{and}\quad\inf_{B_\nu\ti(-\nu^2, R_\infty^\ast\nu^2]} u_\infty\leq -\nu^{1+\beta_1}\ \mbox{for }\nu\in\le(0, \fr{1}{2\sqrt{2}}\ri].
\end{equation}
As in Case I, since $u_\infty$ is also a $C$-viscosity subsolution and a $C$-viscosity supersolution of (\ref{eq:B21/2}) in $Q'$, it follows that (\ref{eq:B3}) holds.

Now, setting
\[
\eta(x, t):=-4|x|^2-8n\L t,
\]
we observe that 
\[
\eta_t+\P^+(D^2\eta)=0\quad\mbox{in } B_{\fr12}\ti\le(0, \fr{R_\infty^\ast}{8}\ri],\]
\[
\eta\leq -1\quad\mbox{on }\p B_{\fr12}\ti\le(0, \fr{R_\infty^\ast}{8}\ri),\quad\mbox{and}\quad \eta\leq 0\quad\mbox{in }B_{\fr12}\ti \{0\}.
\]
Hence, because $\eta\leq u_\infty$ in $\p_p (B_{\fr12}\ti(0, \fr{R_\infty^\ast}{8}])$, the comparison principle yields $\eta\leq u_\infty$ in $B_{\fr12}\ti(0, \fr{R_\infty^\ast}{8}]$. Moreover, by (\ref{eq:B3}) and (\ref{eq:B2'}), for $\nu\in(0, \fr{1}{2\sqrt{2}}]$, we have
\[
-\nu^{1+\beta_1}\geq\inf_{B_\nu\ti(-\nu^2, R_\infty^\ast\nu^2]}u_\infty=\inf_{B_\nu\ti(0, R_\infty^\ast\nu^2]}u_\infty \geq -4(1+2n\L R_\infty^\ast)\nu^{2}.
\]
Therefore, for $\nu$ small enough such that $4(1+2n\L R_\infty^\ast)\nu^{1-\beta_1}<1$, we obtain a contradiction. 
\end{proof}

Now, we shall show our proof of Lemma \ref{leminductive}. 

\begin{proof}[Proof of Lemma \ref{leminductive}]
We only consider the estimate near $C^+[u]$ because the other one can be shown similarly. 

We fix $(x_0, t_0)\in C^+[u]\cap \O^\e_T$; $(u-\psi)(x_0, t_0)=0$. 
For $\rho\in(0, \e]$, which will be fixed later, denoting $\hat{Q}_s$ by $Q_s^\ast$ with $R^\ast=\fr{T-t_0}{\rho^2}$;
\[
\hat{Q}_s:=B_s\ti\le(-s^2, s^2\wedge\le(\fr{T-t_0}{\rho^2}s^2\ri)\ri] \quad\mbox{for }s\in(0, 1],
\]
and setting
\[
\hat{u}(x, t):=u(x_0+\rho x, t_0+\rho^2 t)-\psi(x_0, t_0)-\langle D\psi(x_0, t_0),  \rho x\rangle\quad
\mbox{for } (x, t)\in \hat{Q}_1,\]
we see that $\hat{u}$ is an $L^p$-viscosity solution of (\ref{E1}) in $\hat{Q}_1$, where $F$, $G$, $f$, $\varphi$, $\psi$ are replaced, respectively, by 
\begin{align}
\hat{F}(x, t, \xi, X)&:=\rho^2F\le(x_0+\rho x, t_0+\rho^2 t, \fr{1}{\rho}\xi, \fr{1}{\rho^2}X\ri),\nonumber\\
\hat{G}(x, t, \xi, X)&:=\hat{F}(x, t, \xi+\rho D\psi(x_0, t_0), X)-\hat{F}(x, t, \xi, X)\nonumber\\
\hat{f}(x, t)&:=\rho^2f(x_0+\rho x, t_0+\rho^2 t),\nonumber\\
\Phi(x, t)&:=\varphi(x_0+\rho x, t_0+\rho^2 t)-\psi(x_0, t_0)-\langle D\psi(x_0, t_0), \rho x\rangle,\label{5.3phi}\\
\Psi(x, t)&:=\psi(x_0+\rho x, t_0+\rho^2 t)-\psi(x_0, t_0)-\langle D\psi(x_0, t_0), \rho x\rangle.\label{5.3psi}
\end{align}
We notice that $\zeta^\ast$ in $G$ of Lemma \ref{lemflat} corresponds to $\rho D\psi(x_0, t_0)$. 
We note that for $(x, t)\in \hat{Q}_1$, $\xi$, $\zeta\in\R^n$, $X$, $Y\in S^n$, 
\begin{equation*}
\begin{cases}
&\hat{F}(x, t, 0, O)=0, \mbox{ and}\\
&\P^-(X-Y)-\hat{\mu}(x, t)|\xi-\zeta|\leq \hat{F}(x, t, \xi, X)-\hat{F}(x, t, \zeta, Y)\\
&\leq \P^+(X-Y)+\hat{\mu}(x, t)|\xi-\zeta|,
\end{cases}
\end{equation*}
where $\hat{\mu}(x, t):=\rho \mu(x_0+\rho x, t_0+\rho^2 t)$. Also, since 
\[
|\hat{G}(x, t, \xi, X)|\leq \rho|D\psi(x_0, t_0)|\hat{\mu}(x, t)\quad\mbox{for }(x, t, \xi, X)\in \hat Q_1\ti\R^n\ti S^n,
\]
setting $\hat{g}^\ast(x, t):=\sup\{|\hat{G}(x, t, \xi, X)| : \xi\in\R^n, X\in S^n\}$, we have
\[
\|\hat g^\ast\|_{L^p(\hat Q_1)}\leq \rho^{1+\beta_0}|D\psi(x_0, t_0)|\|\mu\|_{L^p(\O_T)}.
\]
Standard calculations give
\[
\|\hat{f}\|_{L^{p}(\hat{Q}_1)}\leq \rho^{1+\beta_0}\|f\|_{L^p(\Omega_T)}, \quad \|\hat{\mu}\|_{L^q(\hat{Q}_1)}\leq \rho^{1-\fr{n+2}{q}}\|\mu\|_{L^q(\Omega_T)},\]
and
\[\|\Psi\|_{L^\infty(\hat{Q}_r)}\vee \sup_{(x, t)\in \hat{Q}_r}|\Phi(x, t)-\Phi(0, 0)-\la D\Phi(0, 0), x\ra|\leq \hat C_1 \rho^{1+\beta_1}r^{1+\beta_1}\]
for $r\in(0, 1]$, where $\hat C_1$ is a constant depending only on $\|\varphi\|_{C^{1, \beta_1}(\O_T)}$ and $\|\psi\|_{C^{1, \beta_1}(\O_T)}$.
  Thus, we can choose small $\rho\in(0, \e]$ such that
\begin{equation}\label{pf:A1}
\|\hat{f}\|_{L^p(\hat{Q}_1)}\vee\|\hat{\mu}\|_{L^q(\hat{Q}_1)}\vee\|\hat{g}^\ast\|_{L^p(\hat{Q}_1)}\leq \delta_2
\end{equation}
and 
\begin{equation}\label{pf:A2}
\|\Psi\|_{L^\infty(\hat{Q}_r)}\vee\sup_{(x, t)\in \hat{Q}_r} |\Phi(x, t)-\Phi(0, 0)-\langle D\Phi(0, 0), x\rangle|\leq \delta_2 r^{1+\b_1}
\end{equation}
for $r\in(0, 1]$, 
where $\delta_2=\delta_2(n, \L, \l, p, q, \beta_1)\in(0, 1)$ is from Lemma \ref{lemflat}. By Lemma \ref{paralemlocal}, there exist $\gamma=\gamma(\e)\in(0, 1)$ and $K=K(\e)>0$ such that 
\[
\|u(\cdot, \cdot)-u(x_0, t_0)\|_{L^\infty(Q_r^{+}(x_0, t_0)\cap \O_T)}\leq Kr^\gamma\quad\mbox{for }r\in(0, \e].
\]
Hence, we can choose smaller $\rho>0$, if necessary, such that
\begin{equation}\label{eq:C1B1}
\inf_{\hat{Q}_1}\hat{u}\geq -1.
\end{equation}


Next, in order to show the assertion of Lemma \ref{leminductive}, we notice that it is sufficient to prove that 
\begin{equation}\label{Iteration}
\inf_{\hat{Q}_{\nu^k}}\hat{u}\geq -\nu^{k(1+\beta_2)}\quad\mbox{for }k\geq 0,
\end{equation}
where $\nu\in(0, 1)$ is the constant in Lemma \ref{lemflat}.

For $k=0$, by (\ref{eq:C1B1}), the inequality (\ref{Iteration}) holds for $k=0$. 
Assuming that 
\[
\inf_{\hat{Q}_{\nu^k}}\hat{u}\geq -\nu^{k(1+\beta_2)}, 
\]
we shall prove 
\[\inf_{\hat{Q}_{\nu^{k+1}}}\hat{u}\geq -\nu^{(k+1)(1+\beta_2)}.\]

Setting
\[
 u_k(x, t):=\nu^{-k(1+\beta_2)}\hat{u}(\nu^k x, \nu^{2k}t), 
 \]
 we see that $\inf_{\hat{Q}_1}u_k\geq -1$. We also see that $u_k$ is an $L^p$-viscosity solution of (\ref{E1}), where $F$, $G$, $f$, $\Phi$ and $\Psi$ are replaced, respectively, by 
\begin{align*}
F_k(x, t, \xi, X)&:=\nu^{k(1-\beta_2)}\hat{F}(\nu^k x, \nu^{2k}t, \nu^{k\beta_2} \xi, \nu^{k(\beta_2-1)}X),\\
G_k(x, t, \xi, X)&:=\nu^{k(1-\beta_2)}\hat{G}(\nu^k x, \nu^{2k}t, \nu^{k\beta_2} \xi, \nu^{k(\beta_2-1)}X),\\
f_k(x, t)&:=\nu^{k(1-\beta_2)}\hat{f}(\nu^k x, \nu^{2k}t),\\
\Phi_k(x, t)&:=\nu^{-k(1+\beta_2)}\Phi(\nu^k x, \nu^{2k}t),\\ 
\Psi_k(x, t)&:=\nu^{-k(1+\beta_2)}\Psi(\nu^k x, \nu^{2k}t). 
\end{align*}
By (\ref{pf:A1}), (\ref{pf:A2}) and $\beta_2=\b_0\wedge \b_1$, we notice that 
(\ref{A2}) and (\ref{A3}) hold for ($f_k$, $\mu_k$, $g_{k}^\ast$, $\Phi_k$, $\Psi_k$), where
$g_{k}^\ast(x, t):=\sup \{|G_k(x, t, \xi, X)| : \xi\in\R^n, X\in S^n\}.$
 Hence, it follows from Lemma \ref{lemflat} that 
\[
\inf_{\hat{Q}_\nu} u_k\geq -\nu^{1+\beta_1},
\]
which implies (\ref{Iteration}) for $k+1$. 
\end{proof}

We finally prove the local H\"older continuity for the space derivative of $L^p$-viscosity solutions of (\ref{E0}). In what follows, for the $L^p$-viscosity solution $u\in C(\O_T)$ of (\ref{E0}), we use the notation of $\e$-neighborhood of $C^\pm[u]$ for small $\e>0$;
\[
C^\pm_\e[u]:=\{(x, t)\in \O_T^{3\e} : \dist((x, t), C^\pm[u])<\e\}.
\] 

\begin{proof}[Proof of Theorem \ref{Nearcoin}]
 In order to show the assertion, by Proposition \ref{CKS}, we may suppose that $(y, s)$, $(z, \tau)\in C^\pm_{2r_1}[u]$. Furthermore, in view of Lemma \ref{leminductive}, we may suppose that $\dist((y, s), C^\pm[u])>0$ and $\dist((z, \tau), C^\pm[u])>0$. Without loss of generality, we assume that $\tau\leq s$. Choose $(\hat{y}, \hat{s})$, $(\hat{z}, \hat{\tau})\in C^\pm[u]$ such that $d((y, s), (\hat{y}, \hat{s}))=\dist((y, s), C^\pm[u])$ and $d((z, \tau), (\hat{z}, \hat{\tau}))=\dist((z, \tau), C^\pm[u])$. Thus, we see
 \[
d((y, s), (\hat{y}, \hat{s}))>0\quad\mbox{and}\quad d((z, \tau), (\hat{z}, \hat{\tau}))>0. \]
 
 $\ul{\mbox{Case I: } d((y, s), (z, \tau))<\fr12 d((y, s), (\hat{y}, \hat{s})).}$ 
 Setting 
 \[v(x, t):=u(x, t)-u(\hat{y}, \hat{s})-\la Du(\hat{y}, \hat{s}), x\ra,\] in view of Proposition \ref{CKS}, for any $\beta\in (0, \hat{\beta}\wedge\beta_0)$, we see that 
 \begin{align*}
 |Dv(y, s)-Dv(z, \tau)|\leq C\fr{d((y, s), (z, \tau))^{\beta}}{d((y, s), (\hat{y}, \hat{s}))^{1+\beta}}\|v\|_{L^\infty(Q_{d((y, s), (\hat{y}, \hat{s}))}(y, s))}
\\ 
+Cd((y, s), (z, \tau))^{\beta}d((y, s), (\hat{y}, \hat{s}))^{\beta_0-\beta}. 
 \end{align*}
 This together with 
 \[
 |Dv(y, s)-Dv(z, \tau)|=|Du(y, s)-Du(z, \tau)|,\]
 and
 \[
 \|v\|_{L^\infty(Q_{d((y, s), (\hat{y}, \hat{s}))}(y, s))}\leq \|v\|_{L^\infty(Q^+_{2d((y, s), (\hat{y}, \hat{s}))}(\hat{y}, \hat{s})\cap \O_T)}\leq Cd((y, s), (\hat{y}, \hat{s}))^{1+\beta_2}, \]
 where the last inequality follows from Lemma \ref{leminductive}, 
 yields 
  \begin{align*}
 &|Du(y, s)-Du(z, \tau)|\\ 
 &\leq C d((y, s), (z, \tau))^{\beta}\le(d((y, s), (\hat{y}, \hat{s}))^{\beta_2-\beta}
 +d((y, s), (\hat{y}, \hat{s}))^{\beta_0-\beta}\ri)\\
 &\leq Cd((y, s), (z, \tau))^{\beta}
  \end{align*}
 for $\beta\in (0, \hat{\beta}\wedge\beta_0)\cap (0, \beta_1]$. 
 
  $\ul{\mbox{Case II: } d((y, s), (z, \tau))\geq \fr12 d((y, s), (\hat{y}, \hat{s})).}$ In view of Lemma \ref{leminductive}, since 
  \begin{align*}
d((z, \tau), (\hat{z}, \hat{\tau}))&\leq d((z, \tau), (\hat{y}, \hat{s}))\\
&\leq d((z, \tau), (y, s))+d((y, s), (\hat{y}, \hat{s}))\\
&\leq 3d((z, \tau), (y, s)),
\end{align*}
  we obtain 
  \begin{align*}
  &|Du(y, s)-Du(z, \tau)|\\
  &\leq |Du(y, s)-Du(\hat{y}, \hat{s})|+|Du(\hat{y}, \hat{s})-Du(\hat{z}, \hat{\tau})|+|Du(\hat{z}, \hat{\tau})-Du(z, \tau)|\\
  &\leq C\le\{d((y, s), (\hat{y}, \hat{s}))^{\beta_2}+d((\hat{y}, \hat{s}), (\hat{z}, \hat{\tau}))^{\beta_2}+d((\hat{z}, \hat{\tau}), (z, \tau))^{\beta_2}\ri\}\\
  &\leq Cd((y, s), (z, \tau))^{\beta_2}. 
    \end{align*}
 \end{proof}

 \section{Appendix: Local H\"older continuity of derivatives for elliptic problems}
In this section, we consider the following elliptic bilateral obstacle problems 
\begin{equation}\label{Noeq:1}
\min\{\max\{F(x, Du, D^2u)-f, u-\psi\}, u-\varphi\}=0\quad\mbox{in } \O, 
\end{equation}
where $\O\subset\R^n$ is a bounded domain. Hereafter, under the hypothesis 
\begin{equation}\label{NoApq2}
q\geq p>n, 
\end{equation}
we define $\beta_0\in(0, 1)$ by \[\beta_0=1-\fr n p.\]
Suppose that 
\begin{equation}\label{NoAf}
f\in L^p(\O).
\end{equation}
The structure condition on $F$ is that there exists
\begin{equation}\label{NoAmu}
\mu\in L^q(\O), \quad \mu\geq 0\quad\mbox{in }\O
\end{equation}
such that for $x\in \O$, $\xi$, $\zeta\in\R^n$, $X$, $Y\in S^n$, 
\begin{equation}\label{NoAF}
\begin{cases}
&F(x, 0, O)=0, \mbox{ and}\\
&\P^-(X-Y)-\mu(x)|\xi-\zeta|\leq F(x, \xi, X)-F(x, \zeta, Y)\\
&\leq \P^+(X-Y)+\mu(x)|\xi-\zeta|.
\end{cases}
\end{equation}
For obstacles $\varphi$ and $\psi$, we suppose that 
\begin{equation}\label{NoAob}
  \varphi\leq \psi\quad\mbox{in }\O, \quad\mbox{and}\quad \varphi, \psi\in C^{1, \beta_1}(\O) \quad\mbox{for }\beta_1\in(0, 1). 
\end{equation}

Under the above hypotheses, we prove that the first derivative of $L^p$-viscosity solutions $u$ of (\ref{Noeq:1}) is H\"older continuous with the H\"older exponent $\b_0\wedge \beta_1$ near the coincidence set without assuming that $\varphi<\psi$ in $\O$. 

\begin{lem}
Assume $(\ref{NoApq2})$, $(\ref{NoAf})$, $(\ref{NoAmu})$, $(\ref{NoAF})$ and $(\ref{NoAob})$. 
Then, for small $\e>0$, there exists $\hat C_2=\hat C_2(\e)>0$ such that if $u\in C(\O)$ is an $L^p$-viscosity solution of $(\ref{Noeq:1})$, and if $x_0\in \O^\e$ satisfies
\[u(x_0)=\varphi(x_0)\quad (\mbox{resp., } u(x_0)=\psi(x_0)), \]
then it follows that 
$$
|u(x)-u(x_0)-\la D\varphi(x_0),x-x_0\ra |\leq \hat C_2r^{1+\b_2}$$
$$
\le( \mbox{resp., } |u(x)-u(x_0)-\la D\psi (x_0),x-x_0\ra |\leq \hat C_2r^{1+\b_2}\ri)
$$
for $x\in B_r (x_0)$, where 
$$
\b_2:=\b_0\wedge \b_1. 
$$
In particular, $u$ is differentiable at $x_0$, and 
$$Du(x_0)=D\varphi (x_0)\quad (\mbox{resp., }Du(x_0)=D\psi (x_0)).$$ 
\end{lem}

\begin{proof}
We consider the case when $x_0\in C^-[u]\cap \O^\e$; $(u-\varphi)(x_0)=0$.  
For simplicity of notations, we shall suppose $x_0=0\in C^-[u]\cap \O^\e.$

Let $0<r<\fr{\e}{2}$. 
Setting 
\[
v(x):=u(x)-\varphi (0)-\la D\varphi (0),x\ra\quad\mbox{for }x\in B_{2r},\]
we see that $v$ is an $L^p$-viscosity solution of 
\begin{equation*}
\min\{\max\{F(x, Dv, D^2v)+G(x, Dv, D^2v)-f, v-\Psi\}, v-\Phi\}=0\quad\mbox{in }B_{2r},
\end{equation*}
where
\begin{align*}G(x, \xi, X)&:=F(x, \xi+ D\varphi(0), X)-F(x, \xi, X),\\
\Phi(x)&:=\varphi(x)-\varphi (0)-\la D\varphi (0),x\ra,\\
\Psi(x)&:=\psi(x)-\varphi (0)-\la D\varphi (0),x\ra.\end{align*}
We note that (\ref{NoAob}) yields 
\[
|\Phi(x)|\leq \sigma_1(r), \quad\mbox{and}\quad \Psi(x)\geq \Psi(0)+\la D\Psi(0), x\ra -\sigma_1(r)
\]
for $x\in B_{2r}$, where $\sigma_1(r)=\hat Cr^{1+\b_1}$ for some $\hat C>0$. 
Setting 
\[
v_+:=v\vee \s_1(r),\quad\mbox{and}\quad v_-:= v\wedge \le(\Psi(0)+\la D\Psi(0), x\ra -\sigma_1(r)\ri),
\]
we claim that 
\begin{equation}\label{Noproperty}
v_+\leq v_-+4\s_1(r) \quad\mbox{in }B_{2r}. 
\end{equation}
Indeed, when $v_-=v$, noting that $v\geq\Phi\geq -\sigma_1(r)$ in $B_{2r}$, we have (\ref{Noproperty}). If $v_-\neq v$, then since for $x\in B_{2r}$,
\begin{align*} 
\Psi(0)+\la D\Psi(0), x\ra&=\psi(0)-\varphi(0)+\la D\psi(0)-D\varphi(0), x\ra\\
&=v(x)-u(x)+\psi(0)+\la D\psi(0), x\ra \\
&\geq v(x)-\psi(x)+\psi(0)+\la D\psi(0), x\ra\\
&\geq v(x)-\sigma_1(r),
 \end{align*}
we have (\ref{Noproperty}). 

In view of Proposition \ref{prop:def}, we observe that $v_-+4\sigma_1(r)$ is a nonnegative $L^p$-viscosity supersolution of 
\[\P^+(D^2u)+\mu|Du|+f^-+|D\varphi(0)|\mu=0\quad\mbox{in } B_{2r}.\] 
Thus, by the weak Harnack inequality (cf. Proposition 2.4 in \cite{KT}) and $v(0)=0$, there are $\e_0, C_0>0$ such that 
\[
r^{-\fr{n}{\e_0}}\|v_-+4\sigma_1(r)\|_{L^{\e_0}(B_r)}\leq C_0\le(4\sigma_1(r)+r^{1+\beta_0}\|f^-+|D\varphi(0)|\mu\|_{L^p(B_{2r})}\ri).
\]
Hence, from our choice of $\b_2$, we have
\begin{equation}\label{Noineq:wh}
r^{-\fr{n}{\e_0}}\|v_-+4\sigma_1(r)\|_{L^{\e_0}(B_r)}\leq Cr^{1+\b_2}. 
\end{equation}

In contrast, we see that $v_+$ is a nonnegative $L^p$-viscosity subsolution of 
\[
\P^-(D^2u)-\mu|Du|-f^+-|D\varphi(0)|\mu=0\quad\mbox{in }B_{2r}.\]
Hence, by the local maximum principle (cf. Proposition 2.5 in \cite{KT}) with the above $\e_0>0$, we have 
$$
\sup_{B_{\fr{r}{2}}}\,v_+\leq C_1\le( r^{-\fr n{\e_0}}\| v_+\|_{L^{\e_0}(B_r)}+r^{1+\b_0}\| f^++|D\varphi(0)|\mu\|_{L^p(B_{2r})}\ri),
$$
for some $C_1=C_1(\e_0)>0$. 
This together with (\ref{Noproperty}) and (\ref{Noineq:wh}) implies that 
\[
-Cr^{1+\b_1}\leq u(x)-\varphi(0)-\la D\varphi(0), x\ra\leq Cr^{1+\b_2}\quad\mbox{in }B_{\fr{r}{2}}, \]
which concludes the proof. 
\end{proof}



\end{document}